\documentclass[a4paper,12pt]{amsart}
\usepackage{graphicx}
\usepackage{amsmath}
\usepackage{amssymb}
\usepackage{amsthm}
\usepackage{esint}
\usepackage{mathrsfs}
\usepackage[utf8]{inputenc}
\usepackage{mathtools}
\usepackage[a4paper,margin=2.8cm]{geometry}

\usepackage{color}
\usepackage{accents}

\newcommand{\e}{\varepsilon}
\newcommand{\de}{\partial}
\renewcommand{\H}{\mathbb{H}}
\newcommand{\cerchio}[1]{\accentset{\smash{\raisebox{-0.12ex}{$\scriptstyle\circ$}}}{#1}\rule{0pt}{2.3ex}}
\renewcommand{\ni}{\nu}
\newcommand{\ci}[1]{\mathscr{#1}}
\newcommand{\alfa}{\alpha}
\newcommand{\con}[1]{\overline{#1}}
\newcommand{\N}[1]{\left\lVert#1\right\rVert}
\newcommand{\mi}{\mu}
\newcommand{\C}{\mathbb{C}}
\newcommand{\R}{\mathbb{R}}
\newcommand{\bra}{\left\langle}
\newcommand{\ket}{\right\rangle}

\newtheorem{theorem}{Theorem}[section]
\newtheorem{lem}[theorem]{Lemma}

\theoremstyle{Corollary}
\newtheorem{cor}[theorem]{Corollary}
\newtheorem{prop}[theorem]{Proposition}
\newtheorem{defn}[theorem]{Definition}
\newtheorem{conj}[theorem]{Conjecture}
\newtheorem{lemma}[theorem]{Lemma}

\numberwithin{equation}{section}

\begin{document}

\title[Compactness and Noncompactness for the CR Yamabe Problem]{Compactness in Dimension Five and Equivariant Noncompactness for the CR Yamabe Problem}

\author{Claudio Afeltra}
\address{Université de Montpellier,  
	Institut Montpelliérain  Alexander Grothendieck.
	Place Eug\`ene Bataillon, 
	34090 Montpellier,  France}

\email{claudio.afeltra@umontpellier.fr}

\author{Pak Tung Ho}
\address{Department of Mathematics, Tamkang University Tamsui, New Taipei City 251301, Taiwan}

\email{paktungho@yahoo.com.hk}

\author{Andrea Pinamonti}

\address{Dipartimento di Matematica, Università di Trento\\
Via Sommarive, 14, 38123 Povo TN, Italy}
\email{andrea.pinamonti@unitn.it}



\begin{abstract}
We study compactness and noncompactness phenomena for the CR Yamabe equation on compact strictly pseudoconvex CR manifolds. 
First, in dimension five we establish uniform \emph{a priori} estimates for families of positive solutions of subcritical equations for the conformal CR sub-Laplacian
\[
L_{J}u = u^{p},
\]
with $p$ bounded away from the critical exponent, assuming positivity of the CR Yamabe constant and positivity of the $p$-mass at every point. 
As a consequence, the corresponding set of solutions is precompact in H\"older topologies.
Secondly, we consider the equivariant CR Yamabe problem for a compact subgroup $G$ of pseudo-Hermitian transformations. 
We construct a $G$-invariant CR structure on $S^{3}$, not equivalent to the standard one, for which the associated CR Yamabe equation admits a sequence of $G$-invariant solutions whose maxima diverge, thereby proving noncompactness in the equivariant setting. 
The arguments combine a Pohozaev-type identity in pseudohermitian normal coordinates with a blow-up analysis and Liouville-type classification results on the Heisenberg group.
\end{abstract}

\maketitle

\section{Introduction}
  
Suppose $(M,g_0)$ is a compact $n$-dimensional Riemannian manifold 
of dimension $n\geq 3$. 
As a generalization of the Uniformization Theorem, the \textit{Yamabe problem} 
is to find a metric $g$ conformal to $g_0$ such that 
the scalar curvature $R_g$ of $g$ is constant. 
If we write $g=u^{\frac{4}{n-2}}g_0$ for some $0<u\in C^\infty(M)$, 
then the scalar curvatures of $g_0$ and $g$ are related by 
\begin{equation}\label{4}
-\frac{4(n-1)}{n-2}\Delta_{g_0}u+R_{g_0}u=R_g u^{\frac{n+2}{n-2}},
\end{equation}
where $\Delta_{g_0}$ is the Laplacian with respect to the metric $g_0$. 
In view of (\ref{4}), the Yamabe problem is equivalent to finding 
$0<u\in C^\infty(M)$ such that 
\begin{equation}\label{5}
-\frac{4(n-1)}{n-2}\Delta_{g_0}u+R_{g_0}u=c u^{\frac{n+2}{n-2}}
\end{equation}
for some constant $c$. 
The Yamabe problem was solved in a series of work by Aubin \cite{Aubin}, Trudinger \cite{Trudinger},
and Schoen \cite{Schoen}. In other words, (\ref{5}) has at least one solution. 

When $(M,g_0)$ has negative or zero Yamabe constant
(which corresponds to the case when $c<0$ or $c=0$ in (\ref{5}) respectively), 
it is easy to see that the solution to (\ref{5}) is unique up to normalization. 
In the case when $(M,g_0)$ has positive Yamabe constant
(which corresponds to the case when $c>0$ in (\ref{5})), 
Schoen raised the compactness conjecture in a topics course at Stanford in 1988: 
the set of solutions of the Yamabe equation (\ref{5})
is compact except when the manifold is conformally equivalent to
the standard unit sphere $S^n$. 

The compactness conjecture was proved by Khuri, Marques and Schoen \cite{KMS} 
in dimension $n\leq 24$ (see also 
\cite{Druet,LZ,Li&Zhu,M}
for previous results). It turns out that the compactness conjecture is false when $n\geq 25$; 
the counterexample 
was constructed by Brendle for $n\geq 52$ in \cite{Brendle}
and by Brendle and Marques for $25\leq n\leq 51$
in \cite{Brendle&Marques}. 

Since then, various compactness and noncompactness results have been proved 
in different contexts. For example, 
compactness and noncompactness results were obtained for the Yamabe-type problem on 
manifolds with boundary \cite{Almaraz1,Almaraz2,Almaraz&Wang,Chen&Wu,dS,Disconzi&Khuri,Ghimenti&Micheletti,Ho&Shin1,KMW}.

The Yamabe problem can be also proposed in the context of CR manifolds. 
Given a compact strictly pseudoconvex
CR manifold $(M,J,\theta)$ of real dimension $2n +1$,
the \textit{CR Yamabe problem} is to find a 
contact form conformal to $\theta$ such that its Webster scalar curvature is
constant. This was first introduced by Jerison and Lee in 
\cite{Jerison&Lee3}. 
If we write $\tilde{\theta}=u^{\frac{2}{n}}\theta$ for some $0<u\in C^\infty(M)$, then 
the Webster scalar curvature of $\theta$ and $\tilde{\theta}$ are related by 
\begin{equation}\label{2}
L_Ju=\widetilde{R}u^{1+\frac{2}{n}}
\end{equation}
where $L_J$ is the conformal CR sub-Laplacian given by 
$$L_J=-b_n\Delta_b+R,~~b_n:=2+\frac{2}{n}.$$
Here, $R$ 
(and $\widetilde{R}$) is the Webster scalar curvature of $\theta$
(of $\tilde{\theta}$ respectively).
It follows from (\ref{2}) that 
the CR Yamabe problem is equivalent to finding $0<u\in C^\infty(M)$ such that 
\begin{equation}\label{3}
L_Ju=cu^{1+\frac{2}{n}}
\end{equation}
for some constant $c$. 
After the work of Jerison and Lee in 
\cite{Jerison&Lee1,Jerison&Lee2,Jerison&Lee3}, 
the CR Yamabe problem was studied in 
\cite{CMY,Gamara,Gamara&Yacoub}.

Inspired by the Schoen's compactness conjecture stated above, it is natural 
to consider the compactness and noncompactness of solutions 
to the CR Yamabe equation (\ref{3}). 
The following theorem was proved in \cite{Afeltra} by the first author. 

\begin{theorem}[Theorem 1.1 in \cite{Afeltra}]\label{Afeltra_thm1}
Let $(M,J,\theta)$ be a compact $3$-dimensional strictly pseudoconvex
CR manifold of positive CR Yamabe constant such that, for every $x\in M$, 
its $p$-mass at $x$ is positive, i.e. $m_x>0$. 
Then, for every $\epsilon>0$ and $k\in\mathbb{N}$, there exists a constant $C$ such that 
$$\frac{1}{C}\leq u\leq C,~~\|u\|_{\Gamma^{k,\alpha}}\leq C$$
for every $u\in \cup_{1+\epsilon\leq p\leq 3}\mathcal{M}_p$ and $0<\alpha<1$. Here
$$\mathcal{M}_p=\{u>0: L_Ju=u^p\},$$
where $\Gamma_{k,\alpha}$ is the H\"{o}lder space. 
In particular, $\cup_{1+\epsilon\leq p\leq 3}\mathcal{M}_p$
is compact in the $\Gamma^{k,\alpha}$ topology. 
\end{theorem}

As pointed out in \cite{Afeltra}, 
the assumption that $p$-mass is positive at every point $x\in M$
is difficult to check. 
Fortunately, we have the  
results by Takeuchi in \cite{Takeuchi} and 
by Cheng, Malchiodi, and Yang in \cite{CMY}, which say that 
any embeddable $3$-dimensional CR manifold $M$
which is not CR-equivalent to $S^3$ with the standard CR structure
must have positive $p$-mass at every point $x\in M$. 
Combining these with Theorem \ref{Afeltra_thm1}, we have the following: 

\begin{cor}[Corollary 1.2 in \cite{Afeltra}]\label{Afeltra_cor}
Suppose that $(M,J,\theta)$ is an embeddable $3$-dimensional CR manifold  
which has positive CR Yamabe constant and 
is not CR-equivalent to $S^3$ with the standard CR structure.
Then the statement of Theorem \ref{Afeltra_thm1} holds.  
\end{cor}

In this paper, we prove the following theorem, 
which is the corresponding case of Theorem \ref{Afeltra_thm1} for dimension $5$. 

\begin{theorem}\label{CompactnessTheorem}
Let $(M,J,\theta)$ be a compact $5$-dimensional strictly pseudoconvex
CR manifold of positive CR Yamabe constant such that, for every $x\in M$, 
its $p$-mass at $x$ is positive, i.e. $m_x>0$. 
Then, for every $\epsilon>0$ and $k\in\mathbb{N}$, there exists a constant $C$ such that 
$$\frac{1}{C}\leq u\leq C,~~\|u\|_{\Gamma^{k,\alpha}}\leq C$$
for every $u\in \cup_{1+\epsilon\leq p\leq 2}\mathcal{M}_p$ and $0<\alpha<1$. Here
$$\mathcal{M}_p=\{u>0: L_Ju=u^p\},$$
where $\Gamma_{k,\alpha}$ is the H\"{o}lder space. 
In particular, $\cup_{1+\epsilon\leq p\leq 2}\mathcal{M}_p$
is compact in the $\Gamma^{k,\alpha}$ topology. 
\end{theorem}

Like in the three-dimensional case, the verification of the hypothesis of positivity of the mass can be checked by the Positive Mass Theorems existing in the literature: the one in \cite{Cheng&Chiu&Yang} is valid for spherical CR manifolds verifying an analytical condition in dimension five, the one in \cite{Cheng&Chiu} is valid for five-dimensional spherical spin manifolds. Many conjecture that in dimension higher than three, a CR Positive Mass Theorem without additional hypothesis similar to the Riemannian one should hold, but the problem is currently completely open.

Our proof of Theorem \ref{CompactnessTheorem} is inspired by the work of Marques \cite{M}, in particular we adapt to the CR case the technique of \textit{symmetry estimates}.

The \textit{equivariant CR Yamabe problem} was first introduced and studied by the second author in \cite{Ho1}. 
To state it, we recall that a \textit{CR automorphism} of $(M,J,\theta)$ is a diffeomorphism $f: M \to M$ such that
its differential at any point maps horizontal vector to horizontal vector, i.e. 
$f^*(\ker \theta)\subseteq \ker \theta$.
Note that $f$ is a CR automorphism if and only if
$$f^*\theta=u\theta~~\mbox{ for some }u\in C^\infty(M).$$
Let $Aut_{CR}(M,J, \theta)$ be the group of all CR automorphisms of $(M,J,\theta)$. On the other hand, let
$I(M,J,\theta)$ be the group of all pseudo-Hermitian transformations $f$ of $(M, \theta)$, which preserves
the associated contact Riemannian metric
$$g=\theta\cdot\theta+d\theta\circ J.$$
Note that $I (M, J,\theta)$ is a subgroup of $Aut_{CR}(M,J, \theta)$ (cf. \cite{BK}).

\begin{conj}[Equivariant CR Yamabe problem]\label{conj1} Given a compact strictly pseudoconvex
CR manifold $(M,J,\theta)$ of real dimension $2n +1$, and a compact subgroup $G$ of $I (M, \theta)$, there
exists a $G$-invariant contact form conformal to $\theta$ such that its Webster scalar curvature is
constant.
\end{conj}

Here, a contact form $\tilde{\theta}$
is said to be $G$-invariant if  $f^*\tilde{\theta}=\tilde{\theta}$
for all $f\in G$. We remark that
the classical CR Yamabe problem is the special case of Conjecture \ref{conj1} when 
$G =\{id_M\}$.

In \cite{Afeltra&Pinamonti}, the first and third author proved the following: 

\begin{theorem}[Theorem 1.1 in \cite{Afeltra&Pinamonti}]\label{thm1}
There exists a CR structure on $S^3$, not equivalent to the standard one, 
such that the associated CR Yamabe equation 
$$L_Ju=2u^3$$  
has a set of solutions $\{u_k\}_{k\in\mathbb{N}}$ with $\max u_k\to \infty$. 
\end{theorem}

Here, $L_J=-4\Delta_b+R$ is the conformal CR sub-Laplacian when $n=1$.

It is natural to ask whether the corresponding result to 
Theorem \ref{thm1}
is true in the equivariant case. To answer this, we let $f:S^3\to S^3$ be given by 
$$f(w_1,w_2)=(-w_1,w_2)~~\mbox{ for }(w_1,w_2)\in S^3,$$
and let $G=\{f, id_{S^3}\}$. We have the following: 

\begin{theorem}\label{thm2}
There exists a CR structure on $S^3$ which is $G$-invariant and not equivalent to the standard one, 
such that the associated CR Yamabe equation 
\begin{equation}\label{1}
L_Ju=2u^3
\end{equation}
has a set of $G$-invariant solutions $\{u_k\}_{k\in\mathbb{N}}$ with $\max u_k\to \infty$. 
\end{theorem}

Here, a function $u$ is $G$-invariant if $\phi^*u=u$ for all $\phi\in G$, 
and a CR structure is $G$-invariant if $\phi^*J=J$ for all $\phi\in G$. 
Therefore, Theorem \ref{thm2} proves the noncompactness of the equivariant 
CR Yamabe problem. 
We remark that the noncompactness of the equivariant Yamabe problem and 
the equivariant Yamabe problem with boundary has been obtained 
in \cite{Ho} and in \cite{Ho&Shin2} respectively.

We also remark that the compactness and noncompactness of the equivariant CR Yamabe equation 
can be subtle. Indeed, if we take $H$ to be the group $\{h,id_{S^3}\}$, 
where 
$$f(w_1,w_2)=(-w_1,-w_2)~~\mbox{ for }(w_1,w_2)\in S^3.$$
On the Rossi sphere $(S^3,J_s)$, 
all $H$-invariant  $u$ to the CR Yamabe equation (\ref{1})
can be viewed as 
a solution to the CR Yamabe equation (\ref{1}) 
on the quotient $(S^3/H,J_s)$. It is known that $(S^3/H, J_s)$ 
is embeddable and has positive CR Yamabe constant. In particular, it follows from Corollary \ref{Afeltra_cor} stated above
that the set of all solutions to the CR Yamabe equation (\ref{1})
on $(S^3/H, J_s)$ is compact. 
Hence, the set of all $H$-invariant solutions to the CR Yamabe equation (\ref{1})
on the Rossi sphere $(S^3,J_s)$ is compact.

\medskip
\textbf{Organization of the paper:} 
In Section~\ref{SectionPreliminaries} we collect the basic pseudohermitian and CR-geometric preliminaries, fix notation, and recall the CR Yamabe operator and its main analytic properties. 
In Section~\ref{proof} we prove Theorem  \ref{thm2}. 
Section~\ref{comp} is dedicated to the blow up analysis and the proof of Theorem \ref{CompactnessTheorem}. We conclude the paper with an Appendix containing some technical computations.

\medskip

\textbf{Acknowledgements.}
The authors thank Andrea Malchiodi for several illuminating discussions on the topic of the paper.\\
P.T.H. was supported  by the National Science and Technology Council (NSTC),
Taiwan, with grant Number: 114-2115-M-032 -003 -MY2\\
A.P. is a member of the Istituto Nazionale di Alta Matematica (INdAM), Gruppo Nazionale per l'Analisi Matematica, la Probabilità e le loro Applicazioni (GNAMPA), and is supported by the University of Trento, the MIUR-PRIN 2022 Project \emph{Regularity problems in sub-Riemannian structures}  Project code: 2022F4F2LH and the INdAM-GNAMPA 2025 Project \emph{Structure of sub-Riemannian hypersurfaces in Heisenberg groups}, CUP ES324001950001.\\
C.A. is supported by the French National Research Agency (ANR) project  EINSTEIN-PPF, grants ANR-23-CE40-0010.

\section{Preliminaries}\label{SectionPreliminaries}
We shortly recall some basic notions about CR geometry, referring to the monograph \cite{Dragomir&Tomassini} for a complete introduction.
A CR structure on a $2n+1$-dimensional manifold $M$ is a $n$-dimensional complex subbundle $\ci{H}$ of $TM\otimes\C$ such that $\ci{H}\cap\overline{\ci{H}}=\{0\}$ and $[\Gamma(\ci{H}),\Gamma(\ci{H})]\subset\Gamma(\ci{H})$.
The real part of $\ci{H}$, $H(M)=\mbox{Re}(\ci{H}+\overline{\ci{H}})$, is called Levi distribution, and it carries the natural complex structure $J$ defined by $J(Z+\overline{Z})=i(Z-\overline{Z})$. $H(M)$ and $J$ determine the CR structure.
The CR structure is said nondegenerate if $H(M)$ is a contact distribution; it is said pseudoconvex if for some contact form $\theta$ the bilinear form $L_{\theta}(Z,\con{W}) = -d\theta(Z,\con{W})$ is positive definite.
In the following we will always assume the hypothesis of pseudoconvexity.
In such a case, the choice of a contact form $\theta$ determines a rich geometric structure, including a subriemannian metric on $H(M)$ (which induces a subriemannian distance $d$), a measure, and a connection called the Tanaka-Webster connection.
By contracting twice the associated curvature tensor through the metric, a scalar curvature invariant known as Webster scalar curvature is obtained.
Any other contact form for a given CR structure is of the form $\widetilde{\theta}=u^{\frac{2}{n}}\theta$ for some smooth positive function $u$, and the Webster scalar curvature associated to $\widetilde{\theta}$ is given by the formula
$$\widetilde{R} =u^{-\frac{n+2}{n}}(-b_n\Delta_b + R)u$$
where $\Delta_b=\operatorname{div}\circ\nabla_b$ (where $\nabla_b$ is the subriemannian gradient) is a second order operator known as sublaplacian, and $b_n=2+\frac{2}{n}$. Equivalently
$$\Delta_bu = u_{\alfa\overline{\alfa}} +  u_{\overline{\alfa}\alfa}.$$

The most important CR manifold is the Heisenberg group, which is the Lie group $\mathbb{H}^n=\mathbb{C}^n\times\mathbb{R}$ with the group law
$$(z,t)\cdot (w,s)=(z+w,t+s+2\mbox{Im}(z\overline{w})).$$
$\H^n$ is endowed with the left invariant CR structure $\ci{H}$ generated by the left-invariant vector fields
$$Z_{\alfa}=\frac{\partial}{\partial z^{\alfa}}+i\overline{z}^{\alfa}\frac{\partial}{\partial t}$$
and the contact form
\begin{equation}\label{DefinizioneTheta}
\theta=dt+i\sum_{\alfa=1}^n(z^{\alfa}d\overline{z}^{\alfa}-\overline{z}^{\alfa}dz^{\alfa}).
\end{equation}
Every CR manifold has local coordinates around any point with values in $\H^n$, called pseudohermitian normal coordinates (see \cite{Jerison&Lee2}).
We call $H(\mathbb{H}^1)=\mbox{Re}(\ci{H}+\overline{\ci{H}})$ the Levi distribution associated to this CR structure, and $J_0$ the complex structure on it.
The Reeb vector field corresponding to this contact structure is $T=\frac{\partial}{\partial t}$.
$\H^n$ is endowed with the one parameter group of CR and group automorphisms
$$\delta_{\lambda}(z,t) = (\lambda z,\lambda^2t),$$
for $\lambda\in(0,\infty)$, called dilations. The pseudohermitian measure (which is obviously a Haar measure) satisfies $(\delta_{\lambda})_{\#}dx = \lambda^{2n+2}dx$; for this reason the number $Q=2n+2$ is called homogeneous dimension of $\H^n$.
The generator of the group of dilations is the vector field
\begin{equation}\label{DefinitionXi}
 \Xi=\sum_{\alfa=1}^n(z Z_{\alfa}+\overline{z}Z_{\overline{\alfa}})+2tT,
\end{equation}
where $Z_{\overline{\alfa}}=\overline{Z_{\alfa}}$.
On $\H^n$ the sublaplacian is equal to $\frac{1}{2}\nabla_b=\sum_{\alfa=1}^n(Z_{\alfa}Z_{\overline{\alfa}}+Z_{\overline{\alfa}}Z_{\alfa})$.
The CR Yamabe equation in $\H^n$
\begin{equation}\label{YamabeEquationHn}
 -b_n\Delta_bu = u^{\frac{n+2}{n}}
\end{equation}
has the solution
$$U(z,t) = c_n\frac{1}{\left( t^2 + (1+|z|^2)^2\right)^{n/2}}$$
which geometrically corresponds to the standard contact form of the CR sphere $S^{2n+1}$ pull-backed to $\H^n$ through the Cayley transform, a CR equivalence analogous to the stereographic projection.
Calling $L_x(y)=x^{-1}y$ the left translation, thanks to invariance properties of the sublaplacian
\begin{equation}\label{0.1}
	U_{x,\lambda}=\lambda U\circ\delta_\lambda\circ L_x
\end{equation}
for $x\in\H^n$ and $\lambda\in(0,\infty)$ forms a family of solutions.

Let $S^1(\H^n)$ be the completion of $C^{\infty}_c(\H^n)$ with respect to the product
$$\bra u,v\ket = \int_{\H^n}\nabla_bu\cdot\nabla_bv$$
which, by the Folland-Stein immersions, is subset of $L^{\frac{2Q}{Q-2}}$.

Then the family of solutions $\{U_{x,\lambda}\}_{x\in\H^n,\lambda>0}$ is stable in $S^1(\H^n)$ the following sense.

\begin{theorem}[Lemma 5 from \cite{Malchiodi&Uguzzoni}]\label{TheoremMalchiodiUguzzoni}
 $u\in S^1(\H^n)$ solves the linearized equation of \eqref{YamabeEquationHn} in $U$
 $$-b_n\Delta_b u = \frac{n+2}{n}U^{\frac{2}{n}}u$$
 if and only if there exist $a,\gamma\in\R$ and $\mi_{\alfa}\in\C$ such that
 $$u= a\Xi U + \gamma TU + \sum_{\alfa=1}^n(\mi_{\alfa}Z_{\alfa}U + \overline{\mi_{\alfa}}Z_{\overline{\alfa}}U). $$
\end{theorem}

Similarly to Riemannian geometry, there exists a tensor which characterizes conformal flatness (except in the lowest dimensions).
We recall that a CR manifold is called spherical if it is locally CR equivalent with $\H^n$ (or equivalently to $S^{2n+1}$).
The Chern tensor is defined as the tensor
$$S_{\beta}{}^{\alfa}{}_{\lambda\overline{\sigma}} = R_{\beta}{}^{\alfa}{}_{\lambda\overline{\sigma}} - \frac{1}{n+2}\left(R_{\beta}{}^{\alfa}h_{\lambda\overline{\sigma}} + R_{\lambda}{}^{\alfa}h_{\beta\overline{\sigma}} +\delta_{\beta}{}^{\alfa}R_{\lambda\overline{\sigma}} + \delta_{\lambda}{}^{\alfa}R_{\beta\overline{\sigma}}\right) +$$
$$+\frac{R}{(n+1)(n+2)}\left( \delta_{\beta}{}^{\alfa}h_{\lambda\overline{\sigma}} + \delta_{\lambda}{}^{\alfa}h_{\beta\overline{\sigma}} \right)$$
where $R_{\beta}{}^{\alfa}{}_{\lambda\overline{\sigma}}$ is the curvature tensor, $h_{\lambda\overline{\sigma}}$ is the metric and $R_{\beta}{}^{\alfa}=R_{\beta}{}^{\sigma}{}_{\sigma\overline{\gamma}}h^{\alfa\overline{\gamma}}$ is the pseudohermitian Ricci tensor.

Then the following theorem of Chern and Moser \cite{Chern&Moser} holds (see Section 7.3 in \cite{Ivanov&Vassilev} for a more modern presentation and an alternative proof).

\begin{theorem}
 If $n\ge 2$ then $M$ is spherical if and only if $S=0$.
\end{theorem}

\section{Proof of Theorem \ref{thm2}}\label{proof}
In this section we will work in the first Heisenberg group $\H^1$.
We will consider various CR structures whose Levi distribution coincides with $H(\H^1)$, and we will fix the contact form $\theta$ defined in equation \eqref{DefinizioneTheta}.
Therefore the CR structure will be uniquely determined by the complex structure $J$ on $H(\H^1)$, and we will denote by $\nabla_J$, $\Delta_J$ and so forth the various pseudohermitian quantities when the dependence on $J$ will be relevant.

We define 
$\widehat{f}:\mathbb{H}^1\to\mathbb{H}^1$
by 
$$\widehat{f}(z,t)=(-z,t).$$
Let $\widehat{G}=\{\widehat{f},id_{\mathbb{H}^1}\}$.

Let $X$ be the Hilbert space 
$$X=\{u\in L^4(\mathbb{H}^1): \|\nabla_{J_0} u\|_{L^2}<\infty\}$$
with the inner product
$$\langle u,v\rangle_X=\int_{\mathbb{H}^1}\nabla_{J_0} u\cdot \nabla_{J_0} v.$$
Let $X_{\widehat{G}}$ be
the subspace of $X$ which contains all the $\widehat{G}$-invariant functions in $X$, i.e. 
$$X_{\widehat{G}}=\left\{u\in X: u\mbox{ is $\widehat{G}$-invariant}\right\}.$$
Folland-Stein embedding asserts that there exists $K$ such that
$$\|u\|_{L^4}\leq K\|u\|_X.$$

We define
$$\mathcal{M}=\{U_{x,\lambda}: x\in\mathbb{H}^1,\lambda\in (0,\infty)\}$$
where $U_{x,\lambda}$ is given as in (\ref{0.1}). We also let 
$$\mathcal{E}_{(x,\lambda)}=\mbox{span}\{Z_1U_{x,\lambda}, Z_{\overline{1}}U_{x,\lambda},
TU_{x,\lambda}, \Xi U_{x,\lambda}\}^{\perp}.$$
That is to say, 
\begin{equation*}
\begin{split}
\mathcal{E}_{(x,\lambda)}=\big\{u\in X: \langle u,Z_1U_{x,\lambda}\rangle_X=\langle u,Z_{\overline{1}}U_{x,\lambda}\rangle_X=
\langle u,TU_{x,\lambda}\rangle_X=\langle u,\Xi U_{x,\lambda}\rangle_X=0\big\}.
\end{split}
\end{equation*}
We also let 
$$\mathcal{M}_0=\{U_{x,\lambda}: x\in\mathbb{H}^1_0,\lambda\in (0,\infty)\}$$
where 
$$\mathbb{H}^1_0=\{x=(z,t)\in\mathbb{H}^1: z=0\},$$
and 
$$\mathcal{E}_{(x,\lambda),\widehat{G}}
=\{u\in \mathcal{E}_{(x,\lambda)}: u\mbox{ is $\widehat{G}$-invariant}\}.$$
In general, $U_{x,\lambda}\not\in \mathcal{E}_{(x,\lambda),\widehat{G}}$. 
But $U_{x,\lambda}\in \mathcal{E}_{(x,\lambda),\widehat{G}}$ 
whenever $x\in \mathbb{H}^1_0$. 

Note that the CR Yamabe equation for $J$ is the Euler-Lagrange equation for the functional 
$$\mathcal{I}_J(u)=\int_{\mathbb{H}^1}uL_Ju-\int_{\mathbb{H}^1}u^4$$
in the space $X$.

\begin{prop}\label{prop2.1}
There exists a constant $\alpha$ such that if $J$ is a $\widehat{G}$-invariant 
CR structure on $\mathbb{H}^1$ coinciding with $J_0$ on $\mathbb{H}^1\setminus B_1(0)$
and such that $\|J-J_0\|_{\Gamma^2}\leq \alpha$, 
then, for every $(x,\lambda)\in \mathbb{H}^1_0\times(0,\infty)$, 
there exists a unique $v_{x,\lambda}\in \mathcal{E}_{(x,\lambda),\widehat{G}}$ 
with $\|v\|\lesssim\alpha$ satisfying 
$$\pi_{\mathcal{E}_{(x,\lambda),\widehat{G}}}\big(\nabla\mathcal{I}_J(U_{x,\lambda}+v_{x,\lambda})\big)=0.$$
\end{prop}
\begin{proof}
The proof is exactly the same as the proof of \cite[Proposition 3.2]{Afeltra&Pinamonti}. 
\end{proof}

We have the following: 

\begin{prop}\label{prop2.2}
If $(x,\lambda)\in \mathbb{H}^1_0\times (0,\infty)$
is a critical point of $\mathcal{I}_J(U_{x,\lambda}+v_{x,\lambda})$, 
then $\nabla\mathcal{I}_J(U_{x,\lambda}+v_{x,\lambda})=0$. 
\end{prop}
\begin{proof}
By the Lyapunov-Schmidt method, we can conclude that 
\begin{equation}\label{0.3}
\int_{\mathbb{H}^1}\left(4\langle\phi,U_{x,\lambda}+v_{x,\lambda}\rangle_X+R(U_{x,\lambda}+v_{x,\lambda})\phi
-2(U_{x,\lambda}+v_{x,\lambda})^{3}\phi\right)=0
\end{equation}
for all $\phi\in X_{\widehat{G}}$.   
Let $w\in X$ be the unique solution to 
\begin{equation}\label{0.2}
L_J(w)=2(U_{x,\lambda}+v_{x,\lambda})^3. 
\end{equation}
Since $J$ is  $\widehat{G}$-invariant  and 
$U_{x,\lambda}+v_{x,\lambda}$ is  $\widehat{G}$-invariant, 
we conclude that $w$ is $\widehat{G}$-invariant. 
Putting $U_{x,\lambda}+v_{x,\lambda}-w\in  X_{\widehat{G}}$
into (\ref{0.3}) and using (\ref{0.2}), we obtain 
\begin{equation*}
\begin{split}
0&=\int_{\mathbb{H}^1}\Big(4\langle U_{x,\lambda}+v_{x,\lambda}-w,U_{x,\lambda}+v_{x,\lambda}\rangle_X\\
&\hspace{4mm}+R(U_{x,\lambda}+v_{x,\lambda})(U_{x,\lambda}+v_{x,\lambda}-w)
-2(U_{x,\lambda}+v_{x,\lambda})^{3}(U_{x,\lambda}+v_{x,\lambda}-w)\Big)\\
&=\int_{\mathbb{H}^1}\Big(4\langle U_{x,\lambda}+v_{x,\lambda}-w,U_{x,\lambda}+v_{x,\lambda}\rangle_X+4(U_{x,\lambda}+v_{x,\lambda}-w)\Delta_bw
\Big)\\
&=\int_{\mathbb{H}^1}4\langle U_{x,\lambda}+v_{x,\lambda}-w,U_{x,\lambda}+v_{x,\lambda}-w\rangle_X.
\end{split}
\end{equation*}
From this, we conclude that 
$w=U_{x,\lambda}+v_{x,\lambda}$. 
In particular, it follows from (\ref{0.2}) that 
$U_{x,\lambda}+v_{x,\lambda}$ satisfies 
$$L_J(U_{x,\lambda}+v_{x,\lambda})=2(U_{x,\lambda}+v_{x,\lambda})^3.$$
This proves the assertion.   
\end{proof}

Let $x_k=(0,0,\frac{1}{k})\in\mathbb{H}^1_0$
and let $r_k$, $R_k$, $s_k$ be sequences converging to zero
such that the ball $B_{R_k}(x_k)$ are disjoint. 
Let $J$ be a $\widehat{G}$-invariant CR structure 
coinciding with $J_0$ in $\mathbb{H}^1\setminus\bigcup B_{Ar_k}(x_k)$, 
with $J_{s_k}$ on $B_{r_k}(x_k)$, and with $J_f$ on 
$B_{Ar_k}(x_k)\setminus B_{r_k}(x_k)$, where $|f|\leq s_k$. 

Let us define 
$$\Omega_k=\left\{U_{x,\lambda}: x\in\mathbb{H}^1_0, |x-x_k|<R_k, \frac{\alpha}{R_k}<\lambda<\frac{\beta}{r_k}\right\}\subset\mathcal{M}_0$$
with $\alpha$ and $\beta$ to be chosen later.

Then we want to show that, up to choosing the parameters appropriately, 
for every $k$ there exists a critical point of $\mathcal{I}_J(U_{x,\lambda}+v_{x,\lambda})$
in $\Omega_k$, that is, a $\widehat{G}$-invariant solution to the 
CR Yamabe equation for $J$ that is approximately a bubble centered 
at $x_k$.

Note that Lemmas 4.4-4.7 in \cite{Afeltra&Pinamonti} still hold in our case. 

\begin{proof}[Proof of Theorem \ref{thm2}]
We choose the  $\widehat{G}$-invariant CR structure as above, 
and $r_k=2^{-k}$, $R_k=C 2^{-k}$ with $C$ large enough, 
$\alpha,\beta\gg 1$, $s_k=2^{-2^k}$. 
Thanks to Lemmas 4.4-4.7 in \cite{Afeltra&Pinamonti}, 
we have 
$$\max_{\Omega_k}\mathcal{I}_J(U_{x,\lambda}+v_{x,\lambda,s_k})>\max_{\partial\Omega_k}\mathcal{I}_J(U_{x,\lambda}+v_{x,\lambda,s_k}),$$
and therefore there exists a critical point of $\mathcal{I}_J(U_{x,\lambda}+v_{x,\lambda})$ restricted 
to $\mathcal{M}_0$ in $\Omega_k$. 
By Proposition \ref{prop2.2}, it is a free critical point, and 
therefore a $\widehat{G}$-invariant solution to (\ref{1}). 
Finally, since 
\begin{equation*}
\begin{split}
&|B_{r_k}(x_k)|^{\frac{1}{4}}\max_{B_{r_k}(x_k)}(U_{x,\lambda}+v_{x,\lambda,s_k})\\
&\geq \|(U_{x,\lambda}+v_{x,\lambda,s_k})\|_{L^4(B_{r_k}(x_k))}
\geq \|U_{x,\lambda}\|_{L^4(B_{r_k}(x_k))}-\|v_{x,\lambda,s_k}\|_X,
\end{split}
\end{equation*}
we have 
$\max_{B_{r_k}(x_k)}(U_{x,\lambda}+v_{x,\lambda})\to\infty$. 
This completes the proof of Theorem \ref{thm2}. 
\end{proof}

\section{Compactness}\label{comp}
We recall the Pohozaev identity for CR manifolds (see \cite[Proposition 3.3]{Afeltra}), which is the Pohozaev identity for $\H^n$ by Garofalo and Lanconelli in \cite{GaL} written in pseudohermitian normal coordinates.
In the following, when working in such coordinates, we will use a circle superscript to indicate objects coming from $\H^n$ through these coordinates ($\overset{\circ}{\Delta}_b$, $\overset{\circ}{\nabla}_b$,...).

\begin{prop}\label{Pohozaev}
Let $\overline{x}\in M$ and $u$ be a solution of 
$$-b_n\Delta_bu+Ru=\widetilde{R}u^p.$$
Then, in pseudohermitian normal coordinates around $\overline{x}$, 
the following holds 
\begin{equation*}
\begin{split}
&\int_{B_r(\overline{x})}\Bigg(\frac{1}{b_n}\left(\frac{2n+2}{p+1}-n\right)\widetilde{R}u^{p+1}-\frac{1}{b_n}Ru^2+\frac{1}{b_n(p+1)}
\Xi(\widetilde{R})u^{p+1}\\
&\hspace{12mm}-\frac{1}{2b_n}\Xi(R)u^2-(\Xi u+nu)(\Delta_bu-\overset{\circ}{\Delta}_b u)\Bigg)d\overset{\circ}{V}\\
&=\int_{\partial B_r(\overline{x})}\Bigg(\left(\frac{1}{b_n(p+1)}\widetilde{R}u^{p+1}-\frac{1}{2b_n}Ru^2\right)
\Xi\cdot\overset{\circ}{\nu}+\ci{B}(x,u,\overset{\circ}{\nabla}_bu)\Bigg)d\overset{\circ}{\sigma},
\end{split}
\end{equation*}
where 
\begin{equation*}
\begin{split}
\ci{B}(x,u,\overset{\circ}{\nabla}_bu)
&=n u(x)\overset{\circ}{\nabla}_bu(x)\cdot \overset{\circ}{\nu}_{B_{d(x,0)}}(x)-\frac{1}{2}|\overset{\circ}{\nabla}_bu(x)|^2
\Xi(x)\cdot\overset{\circ}{\nu}_{B_{d(x,0)}}(x)\\
&\hspace{4mm}+\Xi u(x)\overset{\circ}{\nabla}_bu(x)\cdot \overset{\circ}{\nu}_{B_{d(x,0)}}(x)
\end{split}
\end{equation*}
and $B_r$ denotes the ball with respect to Kor\'{a}nyi norm. 
\end{prop}

\subsection{Blow-up analysis}

Let $M$ be a $(2n+1)$-dimensional CR manifold equipped with a pseudohermitian structure $\theta$, 
$p_i$ a sequence with $1<p_i\leq b_n-1$ for all $i$ and $p_i\to b_n-1$ as $i\to\infty$, 
and $u_i\in C^2(M)$ a sequence of positive solutions of 
\begin{equation}\label{3.1}
L_J u_i=\widetilde{R} u_i^{p_i},
\end{equation}
where $\widetilde{R}$ is a positive function of class $C^1$. 
 
\begin{defn}
\emph{A point $\overline{x}\in M$ is called a} blow-up point \emph{if there 
exists a sequence $x_i\to\overline{x}$ such that $M_i=u_i(x_i)\to\infty$.}
\end{defn}

\begin{defn}
\emph{A point $\overline{x}\in M$ is called an} isolated blow-up point \emph{if there 
exist $\overline{r}>0$, a constant $C$, and a sequence $x_i\to \overline{x}$ such that 
$x_i$ is a local maximum of $u_i$, $M_i=u_i(x_i)\to\infty$, and 
$$u_i(x)\leq Cd(x,x_i)^{-\frac{2}{p_i-1}}$$
for every $x\in B_{\overline{r}}(x_i)$.}
\end{defn}

Given an isolated blow-up point $\overline{x}$, we define 
$$\overline{u}_i(r)=\int_{\partial B_1(x_i)}u_i\circ\delta_r d\overset{\circ}{\sigma}$$
in pseudohermitian normal coordinates, and $\overline{w}_i(r)=r^{\frac{2}{p_i-1}}\overline{u}_i(r)$.

\begin{defn}\label{defn3.3}
\emph{An isolated blow-up point $\overline{x}$ is called an} isolate simple 
blow-up point \emph{if there exists $\rho\in (0,\overline{r})$ independent of $i$ 
such that $\overline{w}_i$ has exactly one critical point in $(0,\rho)$.}
\end{defn}

We have the following lemma from \cite{Afeltra}.

\begin{lem}[Lemma 4.4 in \cite{Afeltra}]\label{lem3.1}
 If $\overline{x}$ is an isolated blow-up point, then there exists $C$ such that, for $0<r<\overline{r}/3$, there holds 
 $$\max_{B_{2r}(\overline{x})\setminus B_{r/2}(\overline{x})}u_i\leq C\min_{B_{2r}(\overline{x})\setminus B_{r/2}(\overline{x})}u_i.$$
\end{lem}

Consider the CR Yamabe equation on the Heisenberg group
\begin{equation}\label{3.2}
-b_n\Delta_b u= u^{b_n-1}~~\mbox{ in }\mathbb{H}^n.
\end{equation}
We have the following classification theorem 
proved by Flynn and V\'{e}tois. 

\begin{theorem}[Theorem 1.1 in \cite{FV}]\label{thm3.2}
Let $n\geq 2$ and $u$ be a positive solution to \eqref{3.2} such that
$$u(z,t)\leq C(|z|^2+|t|)^{-\frac{n-2}{2}}~~\mbox{ for all }(z,t)\in\mathbb{H}^n\setminus\{(0,0)\}$$
for some constant $C > 0$. Then $u$ is of the form 
\begin{equation}\label{3.3}
U(z,t)=\frac{c_1}{(t^2+(1+|z|^2)^2)^{\frac{n}{2}}}
\end{equation}
up to the left translation $L_{x_0}(x)=x_0^{-1}x$, for some constant $c_1>0$. 
In particular, when $n=2$, any bounded positive solution to \eqref{3.2} must be 
of the form \eqref{3.3} up to the left translation.
\end{theorem}

In the following, given an isolated blow-up point $x_i\to\overline{x}$, 
in order to study the blow up sequence of functions we do a rescaling 
by defining $M_i=u_i(x_i)$ and 
$$v_i=\frac{1}{M_i}u_i\circ \delta_{M_i^{-\frac{p_i-1}{2}}}\circ L_{x_i}$$
defined on $B_{\overline{r}M_i^{\frac{p_i-1}{2}}}(\overline{x})$. Note that 
$v_i$ satisfies 
$$L_{\theta_i}v_i=(\widetilde{R}\circ\delta_{M_i^{-\frac{p_i-1}{2}}})v_i^{p_i}$$
where $L_{\theta_i}$ is the conformal CR sub-Laplacian
with respect to the rescaled contact form $\theta_i=M_i^{-p_i+1}\left(\delta_{M_i^{-\frac{p_i-1}{2}}}\circ L_{x_i}\right)^*\theta$
on the rescaled CR structure. 

In the following all covariant derivatives applied to $v_i$ are meant with respect 
to this rescaled pseudohermitian structure. 

\begin{prop}\label{prop3.1}
Suppose that $n=2$. 
If $\overline{x}$ is an isolated blow-up point, then for any $R_i\to\infty$, 
$\epsilon_i\to 0$ and $k\in\mathbb{N}$, up to subsequence, in pseudohermitian normal coordinates
around $\overline{x}$, there holds 
$$\left\|\frac{1}{M_i}u_i\left(\delta_{M_i^{-\frac{p_i-1}{2}}}(x_i^{-1}\cdot x)\right)-(U\circ\delta_{\widetilde{R}(0)^{1/2}})(x)\right\|\leq \epsilon_i,$$
where $U$ is defined as in \eqref{3.3}, $M_i=u_i(x_i)$, and 
$$\frac{R_i}{\log M_i}\to 0.$$
\end{prop}
\begin{proof}
We just sketch the proof, since it is very similar to that of \cite[Proposition 4.5]{Afeltra}. 
Using the notation above, we note that $v_i$ satisfies 
\begin{equation*}
\left\{
  \begin{array}{ll}
    L_{\theta_i}v_i=(\widetilde{R}\circ\delta_{M-i^{-\frac{p_i-1}{2}}})v_i^{p_i},\\
    v_i(0)=1,\\
  \nabla_{b,\theta_i}v_i(0)=0,\\
 0<v_i(x)<Cd(x,0)^{-\frac{2}{p_i-1}}.
  \end{array}
\right.
\end{equation*}
Following the argument of the proof of \cite[Proposition 4.5]{Afeltra}, 
we find that, for any $k$ and $R>0$, up to subsequences $v_i$ tends to some limit
$v$ in $C^{k,\alpha}(B_R)$. By a diagonal argument, 
we can pass to subsequences and get $v$ defined in $\mathbb{H}^n$ and is bounded in $\mathbb{H}^n$ satisfying
\begin{equation}\label{3.5}
\left\{
  \begin{array}{ll}
    L_{\overset{\circ}{\theta}}v=\widetilde{R}(0)v^{b_n-1},\\
    v(0)=1,\\
  \nabla_{b,\theta_i}v(0)=0,\\
v>0.
  \end{array}
\right.
\end{equation}
Now Proposition \ref{prop3.1}  follows from Theorem \ref{thm3.2}. 
\end{proof}

\begin{lem}\label{lem3.8}
Let $\overline{x}$ be an isolated simple blow-up point, $R_i\to\infty$, 
and suppose that Proposition \ref{prop3.1} holds for some $\epsilon_i\to 0$. 
Then, given a fixed sufficiently small $\delta>0$, 
there exists $\rho_1\in (0,\rho)$ where $\rho$ is the one from Definition 
\ref{defn3.3} such that
\begin{equation*}
\begin{split}
u_i(x)&\leq CM_i^{-\lambda_i}d(x,x_i)^{-2n+\delta},\\
|(u_i)_{,\alpha}(x)|&\leq CM_i^{-\lambda_i}d(x,x_i)^{-2n-1+\delta}~~\mbox{ for }\alpha=1,2,...,n,\\
|(u_i)_{,\alpha\beta}(x)|&\leq CM_i^{-\lambda_i}d(x,x_i)^{-2n-2+\delta}~~\mbox{ for }\alpha,\beta=1,2,...,n,\\
|(u_i)_{,\alpha\overline{\beta}}(x)|&\leq CM_i^{-\lambda_i}d(x,x_i)^{-2n-2+\delta}~~\mbox{ for }\alpha,\beta=1,2,...,n,\\
|(u_i)_{,0}(x)|&\leq CM_i^{-\lambda_i}d(x,x_i)^{-2n-2+\delta},\\
|(u_i)_{,0\alpha}(x)|&\leq CM_i^{-\lambda_i}d(x,x_i)^{-2n-3+\delta}~~\mbox{ for }\alpha=1,2,...,n,\\
|(u_i)_{,00}(x)|&\leq CM_i^{-\lambda_i}d(x,x_i)^{-2n-4+\delta},
\end{split}
\end{equation*} 
for $R_iM_i^{-\frac{p_i-1}{2}}\leq d(x,x_i)\leq \rho_1$, 
where $\lambda_i=(2n-\delta)\frac{p_i-1}{2}-1$. 
\end{lem}
\begin{proof}
 The proof is analogous to the proof of \cite[Lemma 3.3]{Li&Zhu}. 
\end{proof}

\begin{lem}\label{lem3.9}
In the hypotheses of Lemma \ref{lem3.8}
for $|x|\leq \rho_1 M_i^{\frac{p_i-1}{2}}$, 
the following estimates hold: 
\begin{equation*}
\begin{split}
v_i(x)&\leq CM_i^{\frac{p_i-1}{2}\delta}d(x,x_i)^{-2n},\\
|(v_i)_{,\alpha}(x)|&\leq CM_i^{\frac{p_i-1}{2}\delta}d(x,x_i)^{-2n-1}~~\mbox{ for }\alpha=1,2,...,n,\\
|(v_i)_{,\alpha\beta}(x)|&\leq CM_i^{\frac{p_i-1}{2}\delta}d(x,x_i)^{-2n-2}~~\mbox{ for }\alpha,\beta=1,2,...,n,\\
|(v_i)_{,\alpha\overline{\beta}}(x)|&\leq CM_i^{\frac{p_i-1}{2}\delta}d(x,x_i)^{-2n-2}~~\mbox{ for }\alpha,\beta=1,2,...,n,\\
|(v_i)_{,0}(x)|&\leq CM_i^{\frac{p_1-1}{2}\delta}d(x,x_i)^{-2n-2},\\
|(v_i)_{,0\alpha}(x)|&\leq CM_i^{\frac{p_i-1}{2}\delta}d(x,x_i)^{-2n-3}~~\mbox{ for }\alpha=1,2,...,n,\\
|(v_i)_{,00}(x)|&\leq CM_i^{\frac{p_i-1}{2}\delta}d(x,x_i)^{-2n-4},
\end{split}
\end{equation*}  
\end{lem}
\begin{proof}
This follows from Proposition \ref{prop3.1} and Lemma \ref{lem3.8}. 
\end{proof}

\begin{lem}\label{lem3.10}
If $\overline{x}$ is an isolated blow-up point, then, with the notation of Lemma \ref{lem3.8}, 
in pseudohermitian normal coordinates around $x_i$
$$\left|\int_{B_{\rho_1}(x_i)}(u_i+n\Xi u_i)(\overset{\circ}{\Delta}_bu_i-\Delta_b u_i)\right|\leq
\left\{
  \begin{array}{ll}
    CM_i^{-1+ \delta+o(1)}, & \hbox{if $n=2$;} \\
    CM_i^{-\frac{11}{2}+\frac{2}{3}\delta+o(1)}, & \hbox{if $n=3$;} \\
    CM_i^{1+n-n^2+\frac{2}{n}\delta+o(1)}, & \hbox{if $n\geq 4$.}
  \end{array}
\right.$$
\end{lem}
\begin{proof}
Thanks to  \eqref{DefinitionXi}, \eqref{RelazioniInverseCoordinatePH}, Lemma \ref{lem3.9}, 
and the fact that $v_i$ is real, we have 
\begin{equation*}
\begin{split}
&|v_i+n\Xi v_i|\lesssim v_i+|x||\overset{\circ}{Z}_\alpha v_i|+|x|^2|\overset{\circ}{T} v_i|
\lesssim v_i+|x||(v_i)_{,\alpha}|+|x|^2|(v_i)_{,0}|\\
&\lesssim M_i^{\frac{p_i-1}{2}\delta}(1+|x|)^{-2n}
+|x|M_i^{\frac{p_i-1}{2}\delta}(1+|x|)^{-2n-1}+|x|^2M_i^{\frac{p_i-1}{2}\delta}(1+|x|)^{-2n-2}\\
&\lesssim M_i^{\frac{p_i-1}{2}\delta}(1+|x|)^{-2n}. 
\end{split}
\end{equation*}
Furthermore, using Lemmas \ref{SublaplacianoCoordinate} and \ref{lem3.8}, we have 
\begin{align*}
&|\overset{\circ}{\Delta}_bu_i-\Delta_b u_i|\\
&\lesssim |x||(u_i)_{,\alpha}|+|x|^2|(u_i)_{,0}|+|x|^2|(u_i)_{,\alpha\beta}|
+|x|^2|(u_i)_{,\alpha\overline{\beta}}|+|x|^3|(u_i)_{,0\beta}|+|x|^6|(u_i)_{,00}|\\
&\lesssim M_i\Bigg(M_i^{-\frac{p_i-1}{2}}|x|M_i^{\frac{p_i-1}{2}}|(v_i)_{,\alpha}|
+M_i^{-(p_i-1)}|x|^2M_i^{p_i-1}|(v_i)_{,0}|\\
&\hspace{12mm}+M_i^{-(p_i-1)}|x|^2M_i^{p_i-1}|(v_i)_{,\alpha\beta}|
 +M_i^{-(p_i-1)}|x|^2M_i^{p_i-1}|(v_i)_{,\alpha\overline{\beta}}|\\
&\hspace{12mm}
+M_i^{-3\frac{p_i-1}{2}}|x|^3M_i^{3\frac{p_i-1}{2}}|(v_i)_{,0\alpha}|
+M_i^{-3(p_i-1)}|x|^6M_i^{2(p_i-1)}|(v_i)_{,00}|\Bigg)\circ\delta_{M_i^{\frac{p_i-1}{2}}}\\
&\lesssim M_iM_i^{\frac{p_i-1}{2}\delta}\Big(|x|(1+|x|)^{-2n-1}+
|x|^2(1+|x|)^{-2n-2}+|x|^3(1+|x|)^{-2n-3}\\
&\hspace{12mm}+M_i^{-(p_i-1)}|x|^6(1+|x|)^{-2n-4}\Big)\circ\delta_{M_i^{\frac{p_i-1}{2}}}\\
&\lesssim M_iM_i^{\frac{p_i-1}{2}\delta}\Big(|x|(1+|x|)^{-2n-1}+
M_i^{-(p_i-1)}|x|^6(1+|x|)^{-2n-4}
\Big)\circ\delta_{M_i^{\frac{p_i-1}{2}}}.
\end{align*}
Combining these yields 
\begin{align*}
&\left|\int_{B_{\rho_1}(x_i)}(u_i+n\Xi u_i)(\overset{\circ}{\Delta}_bu_i-\Delta_b u_i)\right|\\
&=M_i^2 M_i^{-\frac{p_i-1}{2}(2n+2)}
\left|\int_{B_{\rho_1M_i^{\frac{p_i-1}{2}}}(0)}(v_i+n\Xi v_i)(\overset{\circ}{\Delta}_bu_i-\Delta_b u_i)\circ\delta_{M_i^{-\frac{p_i-1}{2}}}\right|\\
&\lesssim M_i^2 M_i^{-\frac{p_i-1}{2}(2n+2)}
M_i^{(p_i-1)\delta}
\int_{B_{\rho_1M_i^{\frac{p_i-1}{2}}}(0)}
 (1+|x|)^{-2n}\Bigg(|x|(1+|x|)^{-2n-1}\\
&\hspace{7cm}+
M_i^{-(p_i-1)}|x|^6(1+|x|)^{-2n-4}
\Bigg)\\
&\lesssim \left\{
            \begin{array}{ll}
              M_i^{n+3-(n+1)p_i} M_i^{(p_i-1)\delta}\Big(\log M_i+M_i^{-(p_i-1)(n-2)}\Big), & \hbox{if $n=2$;} \\
              M_i^{n+3-(n+1)p_i} M_i^{(p_i-1)\delta}\Big(M_i^{-(p_i-1)(n-1)}+M_i^{-(p_i-1)}\log M_i\Big), & \hbox{if $n=3$;} \\
              M_i^{n+3-(n+1)p_i} M_i^{(p_i-1)\delta}\Big(M_i^{-(p_i-1)(n-1)}+M_i^{-(p_i-1)(n-2)}\Big), & \hbox{if $n\geq 4$.}
            \end{array}
          \right.
\end{align*}
Note that the last terms are in order of 
$$\left\{
  \begin{array}{ll}
    M_i^{-1+ \delta+o(1)}, & \hbox{if $n=2$;} \\
    M_i^{-\frac{11}{2}+\frac{2}{3}\delta+o(1)}, & \hbox{if $n=3$;} \\
    M_i^{1+n-n^2+\frac{2}{n}\delta+o(1)}, & \hbox{if $n\geq 4$,}
  \end{array}
\right.$$
This proves the assertion. 
\end{proof}

\begin{lem}\label{lem3.11}
In the assumptions of Lemma \ref{lem3.8}, if $\tau_i:=b_n-1-p_i$, 
then 
$$\tau_i= 
    O(u_i(x_i)^{-1+ \delta+o(1)})~~\mbox{ whenever }n=2,$$
and in particular, $u_i(x_i)^{\tau_i}\to 1$. 
\end{lem}
\begin{proof}
Applying the Pohozaev identity of Proposition \ref{Pohozaev}
with respect to the base point $x_i$ with $r=\rho_1$: 
\begin{equation*}
\begin{split}
&\int_{B_{\rho_1}}\Bigg(\frac{1}{b_n}\left(\frac{2n+2}{p_i+1}-n\right)\widetilde{R}u_i^{p_i+1}-\frac{1}{b_n}Ru_i^2+\frac{1}{b_n(p_i+1)}
\Xi(\widetilde{R})u_i^{p_i+1}\\
&\hspace{12mm}-\frac{1}{2b_n}\Xi(R)u_i^2-(\Xi u_i+nu_i)(\Delta_bu_i-\overset{\circ}{\Delta}_b u_i)\Bigg)d\overset{\circ}{V}\\
&=\int_{\partial B_{\rho_1}}\Bigg(\left(\frac{1}{b_n(p+1)}\widetilde{R}u_i^{p_i+1}-\frac{1}{2b_n}Ru_i^2\right)
\Xi\cdot\overset{\circ}{\nu}+\ci{B}(x,u_i,\overset{\circ}{\nabla}_bu_i)\Bigg)d\overset{\circ}{\sigma}.
\end{split}
\end{equation*}
We are going to estimate each term in the above expression. 
By Lemma \ref{lem3.9}, we have  
\begin{align*}
&\int_{B_{\rho_1}}\Xi(\widetilde{R})u_i^{p_i+1}\\
&\lesssim M_i^{p_i+1}M_i^{-(n+1)(p_i-1)}\int_{B_{\rho_1M_i^{\frac{p_i-1}{2}}}}
\Big(M_i^{-\frac{p_i-1}{2}}|x|+M_i^{-(p_i-1)}|x|^2\Big)v_i^{p_i+1}\\
&\lesssim M_i^{-np_i+n+2} \Bigg(\int_{B_{\rho_1M_i^{\frac{p_i-1}{2}}}}
 M_i^{-\frac{p_i-1}{2}}|x|v_i^{p_i+1}+\int_{B_{\rho_1M_i^{\frac{p_i-1}{2}}}}M_i^{-(p_i-1)}|x|^2 v_i^{p_i+1}\Bigg)\\
&\lesssim M_i^{\frac{-(2n+1)p_i+2n+5}{2}} 
M_i^{\frac{p_i^2-1}{2}\delta}\int_{B_{\rho_1M_i^{\frac{p_i-1}{2}}}}|x|(1+|x|)^{-2n(p_i+1)}\\
&\hspace{4mm}
+M_i^{-(n+1)p_i+n+3}
M_i^{\frac{p_i^2-1}{2}\delta}\int_{B_{\rho_1M_i^{\frac{p_i-1}{2}}}}|x|^2(1+|x|)^{-2n(p_i+1)}\\
&\lesssim M_i^{\frac{-(2n+1)p_i+2n+5}{2}} 
M_i^{\frac{p_i^2-1}{2}\delta}M_i^{\frac{p_i-1}{2}(-2np_i+5)}
+M_i^{-(n+1)p_i+n+3}
M_i^{\frac{p_i^2-1}{2}\delta}M_i^{\frac{p_i-1}{2}(-2np_i+6)}\\
&=M_i^{-np_i^2+2p_i+n}M_i^{\frac{p_i^2-1}{2}\delta}. 
\end{align*}
Similarly, by Lemma \ref{lem3.9}, we have 
\begin{align*}
&\int_{B_{\rho_1}}\left(\frac{1}{b_n}R+\frac{1}{2b_n}\Xi(R)\right)u_i^2\lesssim M_i^2 M_i^{-(n+1)(p_i-1)}\int_{B_{\rho_1M_i^{\frac{p_i-1}{2}}}}v_i^2\\
&\lesssim M_i^{n+3-(n+1)p_i}M_i^{(p_i-1)\delta}\int_{B_{\rho_1M_i^{\frac{p_i-1}{2}}}}(1+|x|)^{-4n}\\
&\lesssim
\left\{
  \begin{array}{ll}
    M_i^{n+3-(n+1)p_i}M_i^{(p_i-1)\delta}\log M_i, & \hbox{if $n=2$;} \\
    M_i^{2n+2-2np_i}M_i^{(p_i-1)\delta}, & \hbox{if $n\geq 3$.}
  \end{array}
\right.
\end{align*}
Moreover, thanks to Lemma \ref{lem3.8}, we have 
\begin{align*}
&\int_{\partial B_{\rho_1}}\Bigg(\left(\frac{1}{b_n(p+1)}\widetilde{R}u_i^{p+1}-\frac{1}{2b_n}Ru_i^2\right)
\Xi\cdot\overset{\circ}{\nu}+\ci{B}(x,u_i,\overset{\circ}{\nabla}_bu_i)\Bigg)d\overset{\circ}{\sigma}\\
&=O(M_i^{-2\lambda_i})=O(M_i^{-(2n-\delta)(p_i-1)+2}).
\end{align*}
By Proposition \ref{prop3.1}, there holds 
\begin{align*}
&\int_{B_{\rho_1}}u_i^{p_i+1}
\gtrsim M_i^{\frac{p_i-1}{2}(2n+2)}M_i^{p_1+1}\int_{B_{\rho_1M_i^{\frac{p_i-1}{2}}}}\left(\frac{1}{M_i}u_i\circ\delta_{M_i^{-\frac{p_i-1}{2}}}\right)^{p_i+1}\\
&\gtrsim M_i^{(n+2)p_i-n}\int_{B_{\rho_1M_i^{\frac{p_i-1}{2}}}}\left(U\circ\delta_{\widetilde{R}(0)^{1/2}}\right)^{p_i+1}\\
&\gtrsim M_i^{(n+2)p_i-n}M_i^{\frac{3n}{4}+4-\frac{n}{2}p_i}
=M_i^{-\frac{n}{4}+2+(\frac{n}{2}+2)p_i}\gtrsim 1.
\end{align*}
Now, combining these  with Lemma \ref{lem3.10} and the fact that 
\[\displaystyle\frac{1}{b_n}\left(\frac{2n+2}{p_i+1}-n\right)
=\frac{n^2}{(p_i+1)(2n+2)}\tau_i,\] 
we prove the assertion. 
\end{proof}

\begin{lem}\label{LemmaCSigma}
Suppose that $n=2$. 
In the hypothesis of Lemma \ref{lem3.8}, 
if $\overline{x}$ is an isolated simple blow-up point, 
then for every $\sigma\in (0,\frac{r}{2})$
$$\limsup_{i\to\infty}\max_{\partial B_\sigma} M_iu_i(x)\leq C(\sigma).$$ 
\end{lem}
\begin{proof}
Thanks to Lemma \ref{lem3.1}, 
it is sufficient to prove the statement for $\sigma$ small enough. 
In particular, as in the proof of Proposition \ref{prop3.1}, 
we can suppose that $R>0$. 

Let $x_\sigma$ be such that $d(x_\sigma,x_i)=\sigma$, 
and define $w_i(x)=u_i(x_\sigma)^{-1}u_i(x)$. 
Then $w_i$ satisfies 
\begin{equation}\label{3.6}
 L_\theta w_i=u_i(x_\sigma)^{p_i-1}w_i^{p_i}. 
\end{equation}
Thanks to Lemmas \ref{lem3.1} and \ref{lem3.8}, 
for every compact $K\subset B_{\rho_1}(\overline{x})\setminus\{\overline{x}\}$, 
there exists $C_K$ such that 
$C_K^{-1}\leq w_i\leq C_K$. 
Therefore, applying the regularity theory from \cite[Theorem 2.3]{Afeltra}, 
we can deduce that up to subsequences
$w_i\to w$ in $C^2_{loc}(B_{\rho_1}(\overline{x})\setminus\{\overline{x}\})$, 
and since, by Lemma \ref{lem3.8}, $u_i(x_\sigma)\to 0$, passing to the limit
in (\ref{3.6}), we get that $L_\theta w=0$. 

Since the blow-up is isolated simple, and since Proposition 
\ref{prop3.1}
implies that $r^{\frac{2}{p_i-1}}\overline{u}_i$ has a critical point 
in $(0,R_iM_i^{-\frac{p_1-1}{2}})$
after which it is decreasing, 
$r^{\frac{2}{p_i-1}}\overline{u}_i$ is decreasing 
in $(R_iM_i^{-\frac{p_1-1}{2}}, \rho)$, 
and because 
$$u_i(x_\sigma)^{-1}r^{\frac{2}{p_i-1}}\overline{u}_i(r)=r^{\frac{2}{p_i-1}}\overline{w}_i(r)\to r^{n}\overline{w}(r),$$
$r^{n}\overline{w}(r)$ is decreasing on $(0,\rho)$. 
Since $w>0$, $w$ must be singular at $\overline{x}$. 
Corollary 9.1 in \cite{Li&Zhu} can be extended to pseudohermitian geometry 
by repeating the proof with minor adaptations. 
Applying it, we get that 
\begin{equation}\label{3.7}
-\int_{B_\sigma(x_i)}\Delta_b w_i
=-\int_{\partial B_\sigma(x_i)}\nabla_bw_i\cdot \nu=
-\int_{\partial B_\sigma(\overline{x})}\nabla_bw\cdot \nu+o(1)
=c+o(1)>0,
\end{equation} 
while integrating (\ref{3.6}) yields 
\begin{equation}\label{3.8}
-b_n\int_{B_\sigma(x_i)}\Delta_b w_i
=\int_{B_\sigma(x_i)}(-R w_i+u_i(x_\sigma)^{p_i-1}w_i^{p_i})
\leq u_i(x_\sigma)^{-1}\int_{B_\sigma(x_i)}u_i^{p_i}.
\end{equation}
But if we call $r_i=R_iM_i^{-\frac{p_i-1}{2}}$, 
it follows from Proposition \ref{prop3.1}, 
Lemma \ref{lem3.8} and Lemma \ref{lem3.11} that 
\begin{equation}\label{3.9}
\begin{split}
&\int_{B_\sigma(x_i)}u_i^{p_i}
=\left(\int_{B_{r_i}(x_i)}+\int_{B_\sigma(x_i)\setminus B_{r_i}(x_i)}\right)u_i^{p_i}\\
&\lesssim M_i^{-(n+1)(p_i-1)}M_i^{p_i}\int_{B_{R_i}(0)}(1+|x|)^{-2np_i}
+M_i^{-\lambda_ip_i}\int_{B_\sigma(x_i)\setminus B_{r_i}(x_i)}|x|^{(-2n+\delta)p_i}\\
&\lesssim M_i^{-(n+1)(p_i-1)}M_i^{p_i}\log R_i
+M_i^{-\lambda_ip_i}\Big(R_iM_i^{-\frac{p_i-1}{2}}\Big)^{-(2n-\delta)p_i+2n+4}\lesssim M_i^{-1} 
\end{split}
\end{equation}
whenever $n=2$. 
Combining (\ref{3.7})-(\ref{3.9}), 
we see that $M_iu_i(x_\sigma)$ is a bounded sequence. 
Now the assertion follows from Lemma \ref{lem3.1}.
\end{proof}

\begin{prop}\label{LimiteBlowUp}
	If $\overline{x}$ is an isolated simple blow-up point then there exists $C$ such that
	$$M_iu_i(x)\le C d(x,x_i)^{-2n}$$
	if $d(x,x_i)\le\frac{\rho}{2}$.
	Furthermore, up to subsequences, there exists $a>0$ such that
	$$M_iu(x)\to aG_{\overline{x}}(x) + b$$
	in $C^2_{\mathrm{loc}}(B_{\frac{\rho}{2}}(\overline{x})\setminus\{\overline{x}\})$, where $G_{\overline{x}}$ is the Green function of $L_{\theta}$ (which exists because $M$ has positive CR Yamabe class)
	and $L_{\theta}b=0$ on $B_{\frac{\rho}{2}}(\overline{x})$.
\end{prop}

\begin{proof}
	If this were not the case, then, up to subsequences, there would exist a sequence $\widetilde{x}_i$ with $d(x_i,\widetilde{x}_i)\le\frac{\rho}{2}$ and
	\begin{equation}\label{EquazioneDimLimiteBlowUp}
		M_iu_i(\widetilde{x}_i)d(x_i,\widetilde{x}_i)^{2n-2} \to\infty.
	\end{equation}
	Define $\widetilde{r}_i=d(x_i,\widetilde{x}_i)$.
	
	Up to the passage to a subsequence such that Proposition \ref{prop3.1} holds for some $R_i\to\infty$ and $\e_i\le e^{-R_i}$, it is easy to verify, using Lemma \ref{lem3.11} and the fact that
	$\sup_{\lambda>0}\lambda^2U(\delta_{\lambda}(x))\le\frac{C}{|x|^{2n}}$,
	that eventually $\widetilde{r}_i\ge R_iM_i^{-\frac{p_i-1}{2}}$.
	
	Define $\widetilde{u}_i=\widetilde{r}_i^{\frac{2}{p_i-1}}u_i\circ\delta_{\widetilde{r}_i}\circ L_{x_i}$ in $B_2$.
	$\widetilde{u}_i$ satisfies
	$$L_{\theta_i}\widetilde{u}_i = \widetilde{R}\widetilde{u}_i^{p_i}$$
	and verifies the hypotheses of Lemma \ref{LemmaCSigma}, therefore $\max_{\de B_1}\widetilde{u}_i(0)\widetilde{u}_i <\infty$. Using the definition of $\widetilde{u}_i$ and Lemma \ref{lem3.11}, this goes in contradiction with formula \eqref{EquazioneDimLimiteBlowUp}.
	
	Hence $M_iu_i$ is locally bounded in $B_{\frac{\rho}{2}}(\overline{x})\setminus\{\overline{x}\}$, and satisfies
	$$L_{\theta}(M_iu_i) = M_i^{1-p_i}(M_iu_i)^{p_i},$$
	therefore, applying regularity theory,
	$$M_iu_i\to v \;\;\;\text{in}\;\;\; C^2_{\mathrm{loc}}(B_{\frac{\rho}{2}}(\overline{x})\setminus\{\overline{x}\})$$
	with $v$ satisfying $L_{\theta}v=0$.
	Known results about singular solutions (see for example Proposition 9.1 in \cite{Li&Zhu}, which can adapted without difficulty to pseudohermitian geometry) imply the rest of the thesis, except the fact that $a>0$.
	This can be proved by proving that $v$ must be singular with the same proof of Lemma \ref{LemmaCSigma}.
\end{proof}

In order to proceed we need a generalization of Theorem \ref{TheoremMalchiodiUguzzoni}.

\begin{lemma}\label{GeneralizedMUlemma}
	If $u$ is a function on $\H^n$ verifying
	$$-\Delta_b u = U^{\frac{Q+2}{Q-2}}u$$
	and $\lim_{x\to\infty}u=0$, then $u$ is one of the Malchiodi-Uguzzoni solutions of Theorem \ref{TheoremMalchiodiUguzzoni}.
\end{lemma}

\begin{proof}
	The Cayley transform $v$ of $u$ verifies
	$$-\Delta_b v = U^{\frac{Q+2}{Q-2}}v$$
	and $\lim_{x\to 0}|x|^{Q-2}v=0$. In particular $v\in L^1_{loc}(\H^n)$, and therefore it is a distribution. So $-\Delta_b v - U^{\frac{Q+2}{Q-2}}v$ Is a distribution on $\H^n$ whose restriction to $\H^n\setminus\{0\}$ is zero, and therefore it is equal to $\sum_{\alfa}a^{\alfa}\de_{\alfa}\delta$.
    Hence if $\phi$ is a smooth compactly supported function which is one in a neighborhood of the origin,
	$$-\Delta_b (\phi v)- U^{\frac{Q+2}{Q-2}}\phi v = \psi + \sum_{\alfa}a^{\alfa}\de_{\alfa}\delta $$
	for a smooth compactly supported function $\psi$, therefore
	$$\phi v  = G*(U^{\frac{Q+2}{Q-2}}\phi v)  + G*\psi + \sum_{\alfa}a^{\alfa}(-1)^{|\alfa|}\de_{\alfa}G. $$
	Since $|U^{\frac{Q+2}{Q-2}}\phi v|\lesssim \frac{1}{|x|^{Q-2}}$, $|G*(U^{\frac{Q+2}{Q-2}}\phi v)|\lesssim\frac{1}{|x|^{Q-4}}$ for $n\ge 2$ and or $|G*(U^{\frac{Q+2}{Q-2}}\phi v)|\lesssim|\log |x||$ for $n=1$.
	This and the fact that $\lim_{x\to 0}|x|^{Q-2}v=0$ imply that $a^{\alfa}=0$. Therefore $-\Delta_b v = U^{\frac{Q+2}{Q-2}}v$ on $\H^1$ and by regularity theorem it is smooth, therefore the result by Malchiodi and Uguzzoni applies.
\end{proof}

From now on $2n+1=5$.

\begin{lemma}
	If $\overline{x}$ is an isolated simple blow-up point then there exists $\gamma>0$ such that
	$$|v_i-U|\lesssim\max\{M_1^{-2}\log M_i,\tau_i\}$$
	for $|x|\le\gamma M_i^{\frac{p_i-1}{2}}$ (where $\tau_i=\frac{n+2}{n}-p_i$).
\end{lemma}

\begin{proof}
	Let $\ell_i=\gamma M_i^{\frac{p_i-1}{2}}$ and
	$$\Lambda_i=\max_{|y|\le\ell_i}|v_i(y)-U(y)|$$
	realized by some $y_i$.
	If $|y_i|\ge c\ell_i$ for some $c$ then, since $w_i\lesssim U$ and $U(y)\lesssim\frac{1}{|y|^{-2n}}$ on the considered domain, then
	$$|v_i(y_i)-U(y_i)| \lesssim\frac{1}{|y_i|^{-2n}}\lesssim\ell_i^{-2n}\lesssim M_i^{-2}$$
	(using the lemma...) and therefore by definition of $y_i$ we get the thesis.
	Therefore we can suppose that $|y_i|\le\frac{\ell_i}{2}$.
	Define
	$$w_i(x)= \frac{1}{\Lambda_i}(v_i(x)-U(x)).$$
	Then $w_i$ satisfies
	$$L_{\widetilde{\theta_i}}w_i(x) = \frac{1}{\Lambda_i}v_i^{p_i}(x)-\frac{1}{\Lambda_i}U(x)^{\frac{n+2}{n-2}} +\frac{1}{\Lambda_i}(L_{\widetilde{\theta_i}}-\Delta_{\theta_{\H^2}})U = $$
	$$= \frac{1}{\Lambda_i}(v_i^{p_i}-U^{p_i}) + \frac{1}{\Lambda_i}\left(U^{p_i}-U^{\frac{n+2}{n-2}} + (L_{\widetilde{\theta_i}}-\Delta_{\theta_{\H^2}})U\right) =$$
	
	$$= \frac{v_i^{p_i}-U^{p_i}}{v_i-U}w_i +$$
	$$+ \frac{1}{\Lambda_i}\left(U^{p_i}-U^{\frac{n+2}{n-2}} + M_i^{-(p_i-1)}R\circ\delta_{M_i^{-\frac{p_i-1}{2}}}U + O\left(M_i^{-(p_i-1)}(1+|x|)^{-2}\right) \right)$$
	(where, we recall, $\theta_i=M_i^{-(p_i-1)}\left(\delta_{M_i^{-\frac{p_i-1}{2}}}\circ L_{x_i}\right)^*\theta$).
	
	Calling $b_i=\frac{v_i^{p_i}-U^{p_i}}{v_i-U}$ by the estimates of the preceding lemmas $|b_i(x)|\lesssim(1+|x|)^{-4}$ for $|x|\le\ell_i$.
	Call also $Q_i= L_{\widetilde{\theta_i}}w_i-b_iw_i$.
	
	Then
	$$w_i(y)=\int_{B_{\ell_i}}G_i(y,\xi)\left(b_i(\xi)w_i(\xi)-Q_i(\xi)\right)d\xi -\int_{\de B_{\ell_i}}\ni_{\xi} G_i$$
	
	Let $t_i=M_i^{-2}\log M_i$. By contradiction suppose that $\Lambda_i^{-1}\max\{t_i,\tau_i\}\to 0$.
	In CR normal coordinates $R=O(|x|^2)$ (see \cite{Jerison&Lee2}), then
	$$|Q_i(y)|\lesssim \Lambda_i^{-1}\left(\tau_i\log U(1+|x|)^{-8} + M_i^{-\frac{2}{3}}|x|^2(1+|x|)^{-6} + O\left(M_i^{-4/3}(1+|x|)^{-2}\right)\right)$$
	
	By estimates we saw $|w_i|\lesssim\Lambda_i^{-1}M_i^{-2}$, and it is standard that $|G_i(y-\xi)|\lesssim |y-\xi|^{-4}$ for $|y-\xi|\le\frac{\ell_i}{2}$, therefore
	\begin{equation}\label{Estimatew_i}
		|w_i(y)|\lesssim \left((1+|y|)^{-2} +\Lambda_i\log M_i M_i^{-2}\right).
	\end{equation}
	By Schauder estimates up to subsequences $w_i\to w$ in the subriemannian Hölder space
    satisfying
	$$L_{\theta_{\H^2}}w = Uw$$
	and $|w(y)|\lesssim(1+|y|)^{-2}$. By Lemma \ref{GeneralizedMUlemma}, $w$ is one of the Malchiodi-Uguzzoni solutions. The definition of $v$ implies that $w(0)=dw(0)=0$, and therefore $w=0$, which implies (since $w_i(y_i)=1$) that $|y_i|\to\infty$, but this goes against equation \eqref{Estimatew_i} since we had supposed by contradiction that $\Lambda_i\log M_i M_i^{-2}\to 0$.
\end{proof}

\begin{lemma}
	$$\tau_i\lesssim\log M_i M_i^{-2}.$$
\end{lemma}

\begin{proof}
	If this were not true then by the former lemma $|v_i-U_0|\lesssim\tau_i$.
	Calling
	$$w_i(x)= \frac{1}{\tau_i}(v_i(x)-U(x)),$$
	$w_i$ satisfies
	$$L_{\widetilde{\theta_i}}w_i(x) = $$
	$$= \frac{v_i^{p_i}-U^{p_i}}{v_i-U}w_i + \widetilde{Q}_i$$
	where
	$$\widetilde{Q}_i = \frac{1}{\tau_i}O\left(\tau_i\log U(1+|x|)^{-8} + M_i^{-\frac{2}{3}}|x|^2(1+|x|)^{-6} +M_i^{-(p_i-1)}(1+|x|)^{-2}\right).$$
	Suppose by contradiction that $\log M_i M_i^{-2}\tau_i^{-1}\to 0$.
	By Schauder estimates $w_i\to w$ on compact sets. Let $\psi=\frac{Q-2}{2}U+\Xi U$.
	Then we have
	$$\int_{|y|\le\ell_i/2}\psi\frac{1}{\tau_i}O\left(M_i^{-\frac{2}{3}}|x|^2(1+|x|)^{-6} +M_i^{-4/3}(1+|x|)^{-2}\right) \to 0$$
	therefore
	$$\lim_{i\to\infty}\int_{|y|\le\ell_i/2}\psi\widetilde{Q}_i = \int_{\H^2}\psi\log U U^2.$$
	
	At the same time
	$$\int_{|y|\le\ell_i/2}\psi\widetilde{Q}_i = \int_{|y|\le\ell_i/2}\psi(L_{\widetilde{\theta}_i}w_i+b_iw_i)=$$
	$$= \int_{|y|\le\ell_i/2}(L_{\widetilde{\theta}_i}\psi+b_i\psi)w_i + \int_{|y|=\ell_i/2}\left(\psi\ni w_i -w_i\ni\psi \right)$$
	therefore passing to the limit
	$$\lim_{i\to\infty}\int_{|y|\le\ell_i/2}\psi\widetilde{Q}_i = \int_{\H^2}(L_{\H^2}\psi+U^2\psi)w =0$$
	which is a contradiction.
\end{proof}

\begin{cor}
	If $\overline{x}$ is an isolated simple blow-up point then there exists $\gamma>0$ such that
	$$|v_i-U|\lesssim M_1^{-2}\log M_i$$
	for $|x|\le\gamma M_i^{\frac{p_i-1}{2}}$ (where $\tau_i=\frac{n+2}{n}-p_i$).
\end{cor}

Arguing in a similar way as in the previous lemmas, following the proof of Proposition 5.5 and the following Remark 1 in \cite{M}, we get also the following estimates.

\begin{lemma}
	In the previous hypotheses,
	$$|v_i-U|(y) \lesssim M_i^{-3/2}(1+|y|)^{-1}$$
	$$|\nabla_b(v_i-U)|(y) \lesssim M_i^{-3/2}(1+|y|)^{-2}$$
	$$|\nabla_b^2(v_i-U)|(y) \lesssim M_i^{-3/2}(1+|y|)^{-3}$$
\end{lemma}

\begin{lemma}
 If $\overline{x}$ is an isolated simple blow-up point then $S(\overline{x})=0$ (where $S$ is the Chern tensor defined in Section \ref{SectionPreliminaries}).
\end{lemma}

\begin{proof}
 We want to apply the Pohozaev identity from Proposition \ref{Pohozaev}.
	Now
	$$A_i(r) := M_i^2\int_{B_r}\left(-\frac{1}{3}\left(R+\frac{1}{2}\Xi(R)\right)u_i^2 -\left(\Xi u_i+2u_i\right)(\Delta_bu_i(x)-\cerchio{\Delta}_bu_i)\right)d\cerchio{V}=$$
	$$=M_i^2M_i^{2-2(p_i-1)}\int_{B_{rM_i^{\frac{p_i-1}{2}}}}\left(-\frac{1}{3}M_i^{-(p_i-1)}\left(R+\frac{1}{2}\Xi(R)\right)\circ\delta_{-M_i^{\frac{p_i-1}{2}}}v_i^2+ \right.$$
	$$\left. -\left(\Xi v_i+2v_i\right)(\Delta_b'v_i-\cerchio{\Delta}_bv_i)\right)d\cerchio{V}$$
	Defining
	$$\hat{A}_i(r)=M_i^2M_i^{2-2(p_i-1)}\int_{B_{rM_i^{\frac{p_i-1}{2}}}}\left(-\frac{1}{3}M_i^{-(p_i-1)}\left(R+\frac{1}{2}\Xi(R)\right)\circ\delta_{-M_i^{\frac{p_i-1}{2}}}U^2 + \right.$$
    $$\left. -\left(\Xi U+2U\right)(\Delta_b'U-\cerchio{\Delta}_bU)\right)d\cerchio{V}$$
	then
	$$|A_i(r)-\hat{A}_i(r)| \lesssim M_i\int_{B_{rM_i^{\frac{p_i-1}{2}}}}\left(|v_i-U|(y)(1+|y|)^{-4} +\right.$$
	$$\left.+ |\nabla_b(v_i-U)|(y)(1+|y|)^{-3} + |\nabla^2_b(v_i-U)|(y)(1+|y|)^{-2}\right) \lesssim$$
	$$\lesssim M_i^{-1/2}\int_{B_{rM_i^{\frac{p_i-1}{2}}}}(1+|y|)^{-5} \lesssim 1.$$
	Finally
	$$\int_{B_r}\left(\frac{Q}{p+1}-\frac{Q-2}{2}\right)\widetilde{R}u^{p+1} \ge 0$$
	and therefore, since $|M_i^2\ci{B}(x,u,\cerchio{\nabla}_bu_i)|\lesssim 1$ by the fact that $M_iu_i$ has limit in $\Gamma^{2,\alfa}$ in the compacts of $M\setminus\{x\}$, we get that $\hat{A}_i(r)\lesssim 1$.
	Now
	$$M_i^2M_i^{2-2(p_i-1)}\int_{B_{rM_i^{\frac{p_i-1}{2}}}}\left(\Xi U+2U\right)(\Delta_b'U-\cerchio{\Delta}_bU)d\cerchio{V}\lesssim$$
	$$\lesssim M_i^2M_i^{2-2(p_i-1)}M_i^{-(p_i-1)}\int_{B_{rM_i^{\frac{p_i-1}{2}}}}(1+|y|)^{-4}(1+|y|)^{-2} \lesssim 1$$ 
	By following the proof of Proposition 4.2 in \cite{Jerison&Lee2}, it can be shown that
	$$M_i^2\int_{B_{rM_i^{\frac{p_i-1}{2}}}}\left(-\frac{1}{3}M_i^{-(p_i-1)}\left(R+\frac{1}{2}\Xi(R)\right)\circ\delta_{-M_i^{\frac{p_i-1}{2}}}U^2\right)d\cerchio{V} \gtrsim |S(x)|^2\log M_i$$
	therefore the thesis follows.
\end{proof}

\begin{lemma}\label{LemmaSegnoB}
	If $\overline{x}$ is an isolated simple blow-up point and $M_iu_i\to h$ in $M\setminus\{\overline{x}\}$ then
	$$\liminf_{r\to 0}\ci{B}(x,h,\cerchio{\nabla}_bh)\ge 0.$$
\end{lemma}

\begin{proof}
	Using the Pohozaev identity it can be proved like in Lemma 4.14 in \cite{Afeltra}, in a manner similar to the proof of the former lemma, that
	$$\liminf_{r\to 0}\ci{B}(x,h,\cerchio{\nabla}_bh)\gtrsim \liminf_{r\to 0}\int_{B_{rM_i^{\frac{p_i-1}{2}}}}(\Xi U +2U)RU$$
	which is $\ge 0$ reasoning as in the proof of Proposition 4.2 in \cite{Jerison&Lee2}
\end{proof}

\begin{lemma}
 Isolated blow-up points are isolated simple.
\end{lemma}

\begin{proof}
 The proof if the same of Proposition 4.15 in \cite{Afeltra} or of Proposition 4.1 in \cite{Li&Zhu}, using Lemma \ref{LemmaSegnoB}.
\end{proof}

\begin{lemma}\label{LemmaBlowUpSet}
 In the hypotheses of Theorem \ref{CompactnessTheorem}, the set of blow-up points is finite and it consists of isolated simple blow-up points.
\end{lemma}

\begin{proof}
 The proof is the same of \cite{Li&Zhu} or \cite{Afeltra}.
\end{proof}

\begin{prop}\label{GreenFunctionExpansion}
 Let $M$ a five-dimensional pseudoconvex CR manifold of positive CR Yamabe class, and let $p\in M$ such that $S(p)=0$ then the Green function $G_p$ of the conformal sublaplacian at $p$ in CR normal coordinates centered at $p$ verifies
 $$G_p(x) = \frac{a}{|x|^4} + A_p + O(|x|)$$
 where $A_p = b m_x$ for some constant $b>0$.
\end{prop}

\begin{proof}
 Since in CR normal coordinates $R_{\alfa\overline{\beta}}(p)=0$ by Proposition 3.12 in \cite{Jerison&Lee2}, then $R_{\beta}{}^{\alfa}{}_{\lambda\overline{\sigma}}(p)=0$. Thanks to this and to Proposition 2.5 in \cite{Jerison&Lee2}, by repeating the proof of Lemma \ref{SublaplacianoCoordinate}, it can be checked that
 $$(\Delta_b-\cerchio{\Delta}_b)\frac{1}{|x|^4}$$
 is locally bounded around the origin, and therefore
 $$G_p(x) = \frac{a}{|x|^4} + A_p + O(|x|).$$

 The last assertion of the theorem is proved by following the proof of Proposition 3.7 in \cite{Cheng&Chiu}.
\end{proof}

\begin{proof}[Proof of Theorem \ref{CompactnessTheorem}]
 By standard arguments through elliptic theory, it suffices to show that the set is bounded in $C^0$.
 
 If $u_i$ violates this, by Lemma \ref{LemmaBlowUpSet} up to subsequences it has a blow-up set $S=\{\overline{x}^1,\ldots,\overline{x}^N\}$ formed by isolated simple blow-up points.
 Up to passage to subsequences we can suppose that $u_i(x^1_i)\le u_i(x^k_i)$ for every $k$, and calling $w_i=u_i(x^1_i)u_i$, by the previous results we can deduce that
 $$w_i(x) \to h(x) = \sum_{k=1}^Na_kG_{\overline{x}^k}(x).$$
 Thanks to the hypothesis and to Proposition \ref{GreenFunctionExpansion} we can get that in CR normal coordinates
 $$h(x) = \frac{c}{|x|^4} + A'$$
 where $A'>0$, but from this and Lemma \ref{LemmaSegnoB} we get a contradiction.
\end{proof}

\appendix
\section{}
Let $M$ be a $2n+1$-dimensional pseudoconvex pseudohermitian manifold and $x\in M$.

\begin{prop}
	In pseudohermitian normal coordinates
	$$\theta^{\alfa} = (1+O(|x|^2))\cerchio{\theta}^{\alfa} + O(|x|^2)\cerchio{\theta}^{\beta} + O(|x|^2)\cerchio{\theta}^{\con{\beta}} + O(|x|)\cerchio{\theta}$$
	$$\theta = (1+O(|x|^2))\cerchio{\theta} + O(|x|^3)\theta^{\beta} + O(|x|^3)\theta^{{\con{\beta}}}.$$
	\begin{equation}\label{omega}
		{\omega_1}^1= O(|x|)\cerchio{\theta}^{\beta} + O(|x|)\cerchio{\theta}^{\con{{\beta}}} + O(|x|)\cerchio{\theta}
	\end{equation}
\end{prop}

\begin{proof}
	It follows from Proposition 2.5 in \cite{Jerison&Lee2}.
\end{proof}

\begin{lemma}\label{LemmaRelazioniPH}
	In pseudohermitian normal coordinates
	\begin{equation}\label{RelazioniCoordinatePH}
		\begin{cases}
			Z_{\alfa}= \cerchio{Z}_{\alfa} + O(|x|^2)\cerchio{Z}_{\beta} +O(|x|^2)\cerchio{Z}_{\con{\beta}}+ O(|x|^3)\cerchio{T}\\
			Z_{\con{\alfa}}= \cerchio{Z}_{\con{\alfa}} +  O(|x|^2)\cerchio{Z}_{\beta} +O(|x|^2)\cerchio{Z}_{\con{\beta}}+ O(|x|^3)\cerchio{T}\\
			T = O(|x|)\cerchio{Z}_{\beta} + O(|x|)\cerchio{Z}_{\con{\beta}} +(1+ O(|x|^2))\cerchio{T}
		\end{cases},
	\end{equation}
	and
	\begin{equation}\label{RelazioniInverseCoordinatePH}
		\begin{cases}
			\cerchio{Z}_{\alfa}= Z_{\alfa} + O(|x|^2)Z_{\beta} +O(|x|^2)Z_{\con{{\beta}}}+ O(|x|^3)T\\
			\cerchio{Z}_{\con{\alfa}} = Z_{\con{\alfa}} + O(|x|^2)Z_{\beta} + (1+ O(|x|^2))Z_{\con{\beta}} + O(|x|^3)T\\
			\cerchio{T} = O(|x|)Z_{\beta} + O(|x|)Z_{\con{\beta}} + (1+O(|x|^2))T
		\end{cases}
	\end{equation}
\end{lemma}

\begin{proof}
	Letting $Z_{\alfa}= a^{\beta}\cerchio{Z}_{\beta}+ b^{\overline{\beta}}\cerchio{Z}_{\overline{\beta}} + c\cerchio{T}$ and applying $\theta^{\beta}$, $\theta^{\overline{\beta}}$ and $\theta$, we get
	$$\begin{cases}
		\delta_{\alfa}^{\beta} = (\delta_{\alfa}^{\beta}+O(|x|^2))a^{\beta} + O(|x|^2)b^{\overline{\beta}} +O(|x|)c,\\
		\delta_{\overline{\alfa}}^{\overline{\beta}} = (1+O(|x|^2))b^{\overline{\beta}} + O(|x|^2)a^{\beta} + O(|x|)c,\\
		0= (1+O(|x|^2))c + O(|x|^3)a^{\overline{\beta}} + O(|x|^3)b^{\beta}
	\end{cases}$$
	respectively.
	The third one implies that $c=O(|x|^3)$, and using this in the other two allows to deduce that $a^{\beta}=\delta_{\alfa}^{\beta}+O(|x|^2)$ and $b^{\overline{\beta}}=O(|x|^2)$.
	Therefore
	$$Z_{\alfa}= \cerchio{Z}_{\alfa} +  O(|x|^2)\cerchio{Z}_{\beta} +O(|x|^2)\cerchio{Z}_{\con{\beta}}+ O(|x|^3)\cerchio{T}.$$
	Letting $T= d^{\beta}\cerchio{Z}_{\beta}+ e^{\con{\beta}}\cerchio{Z}_{{\con{\beta}}} + f\cerchio{T}$ and applying $\theta^{\beta}$, $\theta^{\overline{\beta}}$ and $\theta$, we get
	$$\begin{cases}
		0 = (1+O(|x|^2))d^{\beta} + O(|x|^2)e^{\con{\beta}} + O(|x|)f,\\
		0 = O(|x|^2)d^{\beta} + (1+O(|x|^2))e^{\con{\beta}} + O(|x|)f\\
		1 = (1+O(|x|^2))f + O(|x|^3)d^{\beta} + O(|x|^3)e^{\con{\beta}}.
	\end{cases}$$
	Arguing as before,
	these imply that $d^{\beta}=O(|x|)$, $e^{\con{\beta}}=O(|x|)$ and $f= 1+ O(|x|^2)$.
	Therefore
	$$T = O(|x|)\cerchio{Z}_{\beta} + O(|x|)\cerchio{Z}_{\con{\beta}} +(1+ O(|x|^2))\cerchio{T}.$$
	So we proved the first part of the Lemma.
	
	We can write the formulas we proved as
	\begin{equation*}
		\left(
		\begin{array}{c}
			Z_{\alfa}\\
			Z_{\con{\alfa}}\\
			T
		\end{array}
		\right)
		=
		\left(
		I+
		\left(
		\begin{array}{ccc}
			O(|x|^2) & O(|x|^2) & O(|x|^3) \\
			O(|x|^2) & O(|x|^2) & O(|x|^3) \\
			O(|x|) & O(|x|) & O(|x|^2)
		\end{array}
		\right)
		\right)
		\left(
		\begin{array}{c}
			\cerchio{Z}_{\alfa}\\
			\cerchio{Z}_{\con{\alfa}}\\
			\cerchio{T}
		\end{array}
		\right).
	\end{equation*}
	Applying the Taylor expansion for the matrix inverse $(I+A)^{-1}=I-A+A^2 +O(\N{A}^3)$, we get
\begin{equation*}
\begin{split}
		\left(
		\begin{array}{c}
			\cerchio{Z}_{\alfa}\\
			\cerchio{Z}_{\con{\alfa}}\\
			\cerchio{T}
		\end{array}
		\right)
		=
		\left(
		I+
		\left(
		\begin{array}{ccc}
			O(|x|^2) & O(|x|^2) & O(|x|^3) \\
			O(|x|^2) & O(|x|^2) & O(|x|^3) \\
			O(|x|) & O(|x|) & O(|x|^2)
		\end{array}
		\right)
		+\right.
        \\
		+\left.
		\left(
		\begin{array}{ccc}
			O(|x|^4) & O(|x|^4) & O(|x|^5) \\
			O(|x|^4) & O(|x|^4) & O(|x|^5) \\
			O(|x|^3) & O(|x|^3) & O(|x|^4)
		\end{array}
		\right)
		+O(|x|^3)
		\right)
		\left(
		\begin{array}{c}
			Z_{\alfa}\\
			Z_{\con{\alfa}}\\
			T
		\end{array}
		\right)=
        \\
		=
		\left(
		I+
		\left(
		\begin{array}{ccc}
			O(|x|^2) & O(|x|^2) & O(|x|^3) \\
			O(|x|^2) & O(|x|^2) & O(|x|^3) \\
			O(|x|) & O(|x|) & O(|x|^2)
		\end{array}
		\right)
		\right)
		\left(
		\begin{array}{c}
			Z_{\alfa}\\
			Z_{\con{\alfa}}\\
			T
		\end{array}
		\right).
        \end{split}
	\end{equation*}
	This implies the second part of the thesis.
\end{proof}

\begin{lemma}
	In pseudohermitian normal coordinates
	\begin{equation}\label{RelazioniInverseCoordinatePH2}
		\begin{cases}
			\cerchio{Z}_{\alfa}^2 = Z_{\alfa}^2 + O(|x|)Z_{\beta} +O(|x|)Z_{\con{{\beta}}}+ O(|x|^2)T + O(|x|^2)Z_{\alfa}Z_{\beta} + O(|x|^2)Z_{\alfa}Z_{\con{{\beta}}} + \\
			+ O(|x|^3)Z_{\alfa}T + O(|x|^4)Z_{\beta}Z_{\gamma} + O(|x|^4)Z_{\beta}Z_{\con{\gamma}} + O(|x|^4)Z_{\con{\beta}}Z_{\con{\gamma}} + \\
			+ O(|x|^5)Z_{\beta}T + O(|x|^5)Z_{\con{{\beta}}}T + O(|x|^6)T^2 \\
			
			\cerchio{Z}_{\alfa}\cerchio{Z}_{\beta} = Z_{\alfa}Z_{\beta} + O(|x|)Z_{\gamma} +O(|x|)Z_{\con{{\gamma}}}+ O(|x|^2)T +\\
			+ O(|x|^2)Z_{\alfa}Z_{\gamma} + O(|x|^2)Z_{\alfa}Z_{\con{{\gamma}}} + O(|x|^3)Z_{\alfa}T + \\
			+ O(|x|^2)Z_{\gamma}Z_{\beta} + O(|x|^2)Z_{\con{{\gamma}}}Z_{\beta} + O(|x|^3)TZ_{\beta} + \\
			+ O(|x|^4)Z_{\gamma}Z_{\mi} + O(|x|^4)Z_{\gamma}Z_{\con{\mi}} + O(|x|^4)Z_{\con{\gamma}}Z_{\con{\mi}} + \\
			+ O(|x|^5)Z_{\gamma}T + O(|x|^5)Z_{\con{{\gamma}}}T + O(|x|^6)T^2 \\
			
			\cerchio{Z}_{\alfa}\cerchio{Z}_{\con{\beta}} = Z_{\alfa}Z_{\con{\beta}} + O(|x|)Z_{\gamma} +O(|x|)Z_{\con{{\gamma}}}+ O(|x|^2)T +\\
			+ O(|x|^2)Z_{\alfa}Z_{\gamma} + O(|x|^2)Z_{\alfa}Z_{\con{{\gamma}}} + O(|x|^3)Z_{\alfa}T + \\
			+ O(|x|^2)Z_{\gamma}Z_{\con{\beta}} + O(|x|^2)Z_{\con{{\gamma}}}Z_{\con{\beta}} + O(|x|^3)TZ_{\con{\beta}} + \\
			+ O(|x|^4)Z_{\gamma}Z_{\mi} + O(|x|^4)Z_{\gamma}Z_{\con{\mi}} + O(|x|^4)Z_{\con{\gamma}}Z_{\con{\mi}} + \\
			+ O(|x|^5)Z_{\gamma}T + O(|x|^5)Z_{\con{{\gamma}}}T + O(|x|^6)T^2 \\
			
			\cerchio{Z}_{\con{\alfa}}\cerchio{Z}_{\con{\beta}} = Z_{\con{\alfa}}Z_{\con{\beta}} + O(|x|)Z_{\gamma} +O(|x|)Z_{\con{{\gamma}}}+ O(|x|^2)T +\\
			+ O(|x|^2)Z_{\con{\alfa}}Z_{\gamma} + O(|x|^2)Z_{\con{\alfa}}Z_{\con{{\gamma}}} + O(|x|^3)Z_{\con{\alfa}}T + \\
			+ O(|x|^2)Z_{\gamma}Z_{\con{\beta}} + O(|x|^2)Z_{\con{{\gamma}}}Z_{\con{\beta}} + O(|x|^3)TZ_{\con{\beta}} + \\
			+ O(|x|^4)Z_{\gamma}Z_{\mi} + O(|x|^4)Z_{\gamma}Z_{\con{\mi}} + O(|x|^4)Z_{\con{\gamma}}Z_{\con{\mi}} + \\
			+ O(|x|^5)Z_{\gamma}T + O(|x|^5)Z_{\con{{\gamma}}}T + O(|x|^6)T^2 \\
			
			\cerchio{Z}_{\alfa}\cerchio{T} = Z_{\alfa}T + O(1)Z_{\beta} + O(1)Z_{\con{\beta}} +O(|x|)T+\\
			+O(|x|)Z_{\alfa}Z_{\beta} + O(|x|)Z_{\alfa}Z_{\con{\beta}} + O(|x|^2)Z_{\beta}T + O(|x|^2)Z_{\con{\beta}}T + \\
			+O(|x|^3)Z_{\beta}Z_{\gamma} + O(|x|^3)Z_{\beta}Z_{\con{\gamma}} + O(|x|^3)Z_{\con{\beta}}Z_{\con{\gamma}} + O(|x|^3)T^2\\
			
			\cerchio{T}^2 = (1+O(|x|^2))T^2+ O(1)Z_{\beta} + O(1)Z_{\con{\beta}} + O(|x|)T + O(|x|)Z_{\beta}T +\\
			+O(|x|)Z_{\con{\beta}}T +O(|x|^2)Z_{\beta}Z_{\gamma} + O(|x|^2)Z_{\beta}Z_{\con{\gamma}} + O(|x|^2)Z_{\con{\beta}}Z_{\con{\gamma}} \\
		\end{cases}
	\end{equation}
\end{lemma}

\begin{proof}
	It follows from Lemma \ref{LemmaRelazioniPH} and some computations.
\end{proof}

\begin{lemma}\label{StimeDerivateCoordinateCR}
	For any function $f$ in CR normal coordinates around $\overline{x}$
	$$|f_{,\alfa}- \cerchio{Z}_{\alfa}f|\lesssim (|f_{,\beta}|+f_{,\con{\beta}}|)|x|^2 + |f_{,0}||x|^3,$$
	
	$$|f_{,0}- \cerchio{T}f|\lesssim (|f_{,\beta}|+f_{,\con{\beta}}|)|x| + |f_{,0}||x|^2,$$
	
	$$|f_{,\alfa\beta}- \cerchio{Z}_{\alfa}\cerchio{Z}_{\beta}f|\lesssim (|f_{,\gamma}|+|f_{,\con{\gamma}}|)|x| + (|f_{,0}|+|f_{,\alfa\gamma}|+|f_{,\alfa\con{\gamma}}|+|f_{,\gamma\beta}|+|f_{,\con{\gamma}\beta}|)|x|^2 + (|f_{,\alfa 0}|+f_{,0\beta})|x|^3+$$
	$$+(|f_{,\gamma\mi}|+|f_{,\gamma\con{\mi}}|+|f_{,\con{\gamma}\con{\mi}}|)|x|^4 +(|f_{,0\gamma}|+|f_{,0\con{\gamma}}|)|x|^5 + |f_{,00}||x|^6,$$
	
	$$|f_{,\alfa\con{\beta}}- \cerchio{Z}_{\alfa}\cerchio{Z}_{\con{\beta}}f|\lesssim (|f_{,\gamma}|+|f_{,\con{\gamma}}|)|x| + (|f_{,0}|+|f_{,\alfa\gamma}|+|f_{,\alfa\con{\gamma}}|+|f_{,\gamma\con{\beta}}|+|f_{,\con{\gamma}\con{\beta}}|)|x|^2 + (|f_{,\alfa 0}|+f_{,0\con{\beta}})|x|^3+$$
	$$+(|f_{,\gamma\mi}|+|f_{,\gamma\con{\mi}}|+|f_{,\con{\gamma}\con{\mi}}|)|x|^4 +(|f_{,0\gamma}|+|f_{,0\con{\gamma}}|)|x|^5 + |f_{,00}||x|^6,$$

	$$|f_{,0\alfa}- \cerchio{Z}_{\alfa}\cerchio{T}f|\lesssim |f_{,\beta}|+|f_{,\con{\beta}}| + (|f_{,0}|+|f_{,\alfa\beta}|+|f_{,\alfa\con{\beta}}|)|x| +(|f_{,0\beta}| +|f_{,0\con{\beta}}|)|x|^2 +$$
	$$+ (|f_{,00}|+ |f_{,\beta\gamma}|+|f_{,\beta\con{\gamma}}|+|f_{,\con{\beta}\con{\gamma}}|)|x|^3,$$

	$$|f_{,00}- \cerchio{T}^2f|\lesssim |f_{,\beta}|+|f_{,\con{\beta}}| + (|f_{,0}|+|f_{,0\beta}|+|f_{,0\con{\beta}}|)|x| + (|f_{,\beta\gamma}|+|f_{,\beta\con{\gamma}}| + |f_{,\con{\beta}\con{\gamma}}| +|f_{,00}|)|x|^2.$$\\
\end{lemma}

\begin{proof}
	The first two estimates follow from formulas \eqref{RelazioniInverseCoordinatePH} and \eqref{omega}. The other ones from formulas \eqref{RelazioniInverseCoordinatePH2} and \eqref{omega}.
\end{proof}

\begin{lemma}\label{SublaplacianoCoordinate}
	In pseudohermitian normal coordinates around a point $x$
	$$\Delta_bf = \cerchio{\Delta}_bf + O(|x|)(|\cerchio{Z}_{\beta}f|+|\cerchio{Z}_{\con{\beta}}f|)+$$
    $$ + O(|x|^2)(|Tf|+|\cerchio{Z}_{\beta}\cerchio{Z}_{\alfa}f|+|\cerchio{Z}_{\con{\beta}}\cerchio{Z}_{\alfa}f|+|\cerchio{Z}_{,\con{\alfa}}\cerchio{Z}_{,\beta}f|+|\cerchio{Z}_{\con{\alfa}}\cerchio{Z}_{\con{\beta}}f|) +$$
    $$+ O(|x|^3)(|\cerchio{Z}_{\alfa}Tf|+\cerchio{Z}_{\con{\alfa}}Tf) + O(|x|^4)(|\cerchio{Z}_{\gamma}\cerchio{Z}_{\beta}f|+|\cerchio{Z}_{\con{\gamma}}\cerchio{Z}_{\beta}f|+|\cerchio{Z}_{\con{\gamma}}\cerchio{Z}_{\con{\beta}}f|)+ $$
    $$+O(|x|^5)(|\cerchio{Z}_{\beta}Tf|+|\cerchio{Z}_{\con{\beta}}Tf|) + O(|x|^6)|T^2f|$$
\end{lemma}

\begin{proof}
	Using formulas \eqref{RelazioniCoordinatePH} and Lemma \ref{StimeDerivateCoordinateCR} we get
	
	$$|(\Delta_b -\cerchio{\Delta}_b)f| = |f_{\alfa\con{\alfa}}+f_{\con{\alfa}\alfa} - \cerchio{Z}_{\alfa}\cerchio{Z}_{\con{\alfa}}f - \cerchio{Z}_{\con{\alfa}}\cerchio{Z}_{\alfa}f| \lesssim $$
	
	$$\lesssim (|f_{,\beta}|+|f_{,\con{\beta}}|)|x| + (|f_{,0}|+|f_{,\alfa\beta}|+|f_{,\alfa\con{\beta}}|+|f_{,\beta\con{\alfa}}|+|f_{,\con{\beta}\con{\alfa}}|)|x|^2 + (|f_{,\alfa 0}|+f_{,0\con{\alfa}})|x|^3+$$
	$$+(|f_{,\beta\gamma}|+|f_{,\beta\con{\gamma}}|+|f_{,\con{\beta}\con{\gamma}}|)|x|^4 +(|f_{,0\beta}|+|f_{,0\con{\beta}}|)|x|^5 + |f_{,00}||x|^6\lesssim$$
	
	$$\lesssim (|\cerchio{Z}_{\beta}f|+|\cerchio{Z}_{\con{\beta}}f|)|x| + (|Tf|+|\cerchio{Z}_{\beta}\cerchio{Z}_{\alfa}f|+|\cerchio{Z}_{\con{\beta}}\cerchio{Z}_{\alfa}f|+|\cerchio{Z}_{,\con{\alfa}}\cerchio{Z}_{,\beta}f|+|\cerchio{Z}_{\con{\alfa}}\cerchio{Z}_{\con{\beta}}f|)|x|^2 + (|\cerchio{Z}_{\alfa}Tf|+\cerchio{Z}_{\con{\alfa}}Tf)|x|^3+$$
	$$+(|\cerchio{Z}_{\gamma}\cerchio{Z}_{\beta}f|+|\cerchio{Z}_{\con{\gamma}}\cerchio{Z}_{\beta}f|+|\cerchio{Z}_{\con{\gamma}}\cerchio{Z}_{\con{\beta}}f|)|x|^4 +(|\cerchio{Z}_{\beta}Tf|+|\cerchio{Z}_{\con{\beta}}Tf|)|x|^5 + |T^2f||x|^6$$
\end{proof}

\bibliographystyle{amsplain}

\end{document}